\documentclass[a4paper,12pt]{article}
\usepackage[utf8x]{inputenc}
\usepackage[T1]{fontenc}
\usepackage{amsfonts,amsthm,amssymb,mathrsfs,amsmath}
\usepackage{graphicx,graphics}
\usepackage{relsize}
\usepackage[all]{xy}
\usepackage{caption}
\allowdisplaybreaks

\newcommand{\ZZ}{\mathbb{Z}}
\newcommand{\NN}{\mathbb{N}}
\newcommand{\CC}{\mathbb{C}}
\newcommand{\RR}{\mathbb{R}}
\newcommand{\Hi}{\mathscr{H}}

\newcommand{\dprod}[2]{\left\langle #1,#2\right\rangle}
\newcommand{\norm}[1]{\left\lVert #1\right\rVert}
\newcommand{\proj}[1]{\left|#1\right\rangle\left\langle#1\right|}

\newcommand{\matx}[1]{\left(\begin{matrix} #1 \end{matrix}\right)}

\newtheorem{theorem}{Theorem}[section]
\newtheorem{lemma}[theorem]{Lemma}

\numberwithin{equation}{section}

\newtheorem{corollary}[theorem]{Corollary}

\title{Multiplicity bound of Singular Spectrum for higher rank Anderson models}
\author{Anish Mallick\footnote{e-mail: \texttt{anishm@imsc.res.in} }\\The Institute of Mathematical Sciences, Chennai, India.}
\date{\today}

\begin{document}
\maketitle
\begin{abstract}
In this work, we prove a bound on multiplicity of the singular spectrum for certain class of Anderson Hamiltonians. 
The operator in consideration is of the form $H^\omega=\Delta+\sum_{n\in\ZZ^d}\omega_n P_n$ on the Hilbert space $\ell^2(\ZZ^d)$, where $\Delta$ is discrete laplacian, $P_n$ are projection onto $\ell^2(\{x\in\ZZ^d:n_il_i<x_i\leq (n_i+1)l_i\})$ for some $l_1,\cdots,l_d\in\NN$ and $\{\omega_n\}_n$ are i.i.d real bounded random variables following absolutely continuous distribution. 
We prove that the multiplicity of singular spectrum is bounded above by $2^d-d$ independent of $\{l_i\}_{i=1}^d$. 
When $l_i+1\not\in 2\NN\cup3\NN$ for all $i$ and $gcd(l_i+1,l_j+1)=1$ for $i\neq j$, we also prove that the singular spectrum is simple.
\end{abstract}
\section{Introduction}

Random Schr\"{o}dinger operators and their tight binding version, Anderson model, are well studied for their Spectral properties. Localization and spectral statistics are widely worked on areas in this subject.
It was recently found by Hislop-Krishna\cite{HK1} that the spectral multiplicity plays a role in determining the spectral statistics in the localized Anderson and random Schr\"{o}dinger models.

It becomes therefore important to know the multiplicity of the spectrum, mainly in the random Schr\"{o}dinger case. 
In the Anderson tight binding model, which is the rank one case, Barry Simon\cite{BS2} showed that any standard basis vector $\delta_n$ is cyclic in region of pure point spectrum. 
Other works in pure point regime are by Klein-Molchanov \cite{KM2} and Aizenman-Warzel \cite{AW3}.
Jak\v{s}i\'{c}-Last in \cite{JL1,JL2} showed that the singular spectrum is almost surely simple in case of Anderson type Hamiltonians where rank of the perturbation is one. 

But for higher rank case, spectral simplicity is not always true.
In \cite{SSH} Sadel and Schulz-Baldes worked with certain family of random Dirac operators and showed non-trivial multiplicity of spectrum depending on certain parameter defining the model.
Though Naboko-Nichols-Stolz \cite{NNS} showed simplicity of point spectrum for the operator \eqref{MainOperator} given below, in some particular cases.

 We consider the higher rank Anderson model as first step for studying the random Schr\"{o}dinger case and look at a conjecture of Naboko-Nichols-Stolz (the conjecture is implicit in their paper \cite{NNS}). 
Here we address the question of spectral multiplicity of singular spectrum for a class of higher rank Anderson models.

The Hamiltonian we will work on is like the Anderson tight binding model, except that the perturbations are equal over boxes. The operator can be described as follows.
On the Hilbert space $\ell^2(\ZZ^d)$ the family of operators in consideration are given by
\begin{equation}\label{MainOperator}
 H^\omega=\Delta+\sum_{n\in\ZZ^d} \omega_n P_n,
\end{equation}
\begin{minipage}{0.6\textwidth}
where $\Delta$ and $\{P_n\}_{n\in\ZZ^d}$ are
$$ (\Delta u)(x)=\sum_{i=1}^d u(x+e_i)+u(x-e_i),$$
and
$$(P_n u)(x)=\left\{\begin{matrix} u(x) & n_il_i<x_i\leq (n_i+1)l_i~\forall i \\ 0 & otherwise\end{matrix}\right.,$$
where $\{e_i\}_{i=1}^d$ is the standard generator of the group $\ZZ^d$ and $l_1,\cdots,l_d\in\NN$. The sequence $\{\omega_n\}_{n\in\ZZ^d}$ are i.i.d real bounded random variables following absolutely continuous distribution $\mu$. 
The figure gives a representation for $d=2$.
\end{minipage}
\begin{minipage}{0.35\textwidth}
\centering
\includegraphics[width=2.4in,keepaspectratio]{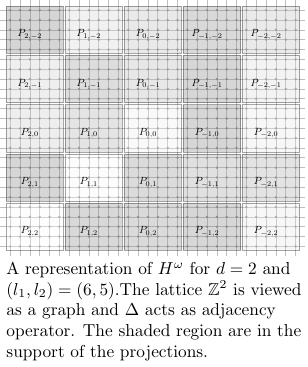}
\end{minipage}

We will view the random variables $\{\omega_n\}_{n\in\ZZ^d}$ as random variables over the probability space $(\Omega,\mathcal{B},\mathbb{P})$ (we can work with $(\RR^{\ZZ^d},\otimes_{\ZZ^d}\mathscr{B}(\RR),\otimes_{\ZZ^d}\mu)$ defined through Kolmogorov construction). So $\omega\mapsto H^\omega$ is a self adjoint operator valued random variable.
The main theorem of this manuscript is:
\begin{theorem}\label{mainthm}
Let $H^\omega$ be described by \eqref{MainOperator} where $l_1,\cdots,l_d\in\NN\setminus\{1\}$, then almost surely
\begin{enumerate}
\item For $d=2$ if $gcd(l_1+1,l_2+1)=1$ then the singular spectrum is simple.
\item For $d>2$ if $l_1,\cdots,l_d\in\NN$ be such that $l_i+1\not\in 2\NN\cup 3\NN$ for all $i$ and $gcd(l_i+1,l_j+1)=1$ for all $i\neq j$, then the singular spectrum is simple.
\item For generic $l_1,\cdots,l_d\in\NN\setminus\{1\}$, the maximum multiplicity of singular spectrum is at most $2^d-d$.
\end{enumerate}
\end{theorem}
The case $d=1$ is not provided here because it follows through properties of Jacobi operator. Note that the bound on multiplicity depends only on $d$.
Following the steps of the proof of the theorem, part (3) of the theorem can be improved to:
\begin{corollary}\label{corMainThm}
Let $d>1$, $l_1,\cdots,l_s\in\NN\setminus\{1\}$ and $l_{s+1}=\cdots=l_d=1$ be given. Then for the operator $H^\omega$ defined by \eqref{MainOperator} we have
\begin{enumerate}
\item For $s=2$ if $gcd(l_1+1,l_2+1)=1$ then the singular spectrum is simple.
\item For $s>2$ if $l_1,\cdots,l_s\in\NN$ be such that $l_i+1\not\in 2\NN\cup 3\NN$ for all $i$ and $gcd(l_i+1,l_j+1)=1$ for all $i\neq j$, then the singular spectrum is simple.
\item For generic $l_1,\cdots,l_s\in\NN\setminus\{1\}$, the maximum multiplicity of singular spectrum is at most $2^s-s$.
\end{enumerate}
\end{corollary}
Notice that for $s=1$, third part gives simplicity for the singular spectrum. 
In this case, proof for simplicity  is much easier than the procedure followed here. See Naboko-Nichols-Stolz \cite{NNS} for a simpler proof in the regime of pure point spectrum. 
Recall that for large disorder the spectrum has a non-empty singular component in view of the localization results already proved by several authors, for example Aizenman-Molchanov\cite{AM}.
\subsection{Ideas of proof}
The proof of the theorem uses properties of Matrix-valued Herglotz functions (see \cite{GT1} for some of their properties) we will focus on the linear map
$$P_0(H^\omega-z)^{-1}P_0:\ell^2(\Lambda(0))\rightarrow\ell^2(\Lambda(0))$$
for $z\in\CC\setminus\sigma(H^\omega)$.  We use the notation $\Lambda(n)=\{x\in\ZZ^d: n_il_i<x_i\leq (n_i+1)l_i~\forall i\}$ to represent the support of the projection $P_n$ in $\ZZ^d$.
The main idea is to show that the multiplicity of singular spectrum is bounded above by the size of a cluster of eigenvalues for a matrix of the form $r^2 P_0\Delta P_0+\sum_{\norm{n}_1=1} (\lambda_n+r) Q_n$, where $Q_n$ are projections onto the faces of the box $\Lambda(0)$ and $r$ is large enough. The matrix arises by taking the first few terms of the Neumann series associated to $P_0(H^\omega-z)^{-1}P_0$.
We will show that the bound on the multiplicity is independent of the perturbation, so we take $\lambda_n$ from intervals chosen appropriately for our calculation.
In this case we can show that the gaps between eigenvalue clusters are large. 
This provides us a way of bounding the number of eigenvalues in each of the clusters from above. Finally for showing simplicity of the spectrum, we only need to show that each cluster has only one point.

The above scheme is implemented in several steps.
First we will show that the maximum multiplicity of singular spectrum is given by essential supremum of \emph{maximum eigenvalue multiplicity for the matrix} $\lim_{\epsilon\downarrow 0}P_0(H^\omega-x-\iota \epsilon)^{-1}P_0$ with respect to Lebesgue measure. 
This is the statement of lemma \ref{multLem1}, and this follows as a consequence of theorem \ref{singSubSpThm} and Poltoratskii's theorem \cite{POL1}. 
Spectral Averaging \cite[Corollary 4.2]{CH1} also plays an important role in its proof. 
Next, in the lemma \ref{multLem2} we show that if we have a bound on the maximum multiplicity of eigenvalue for the matrix $P_0(H^\omega+\lambda P_n-z)^{-1}P_n$ for $\lambda$ in some interval and $z$ in some positive Lebesgue measure set of $\RR$, then the bound holds almost everywhere (w.r.t Lebesgue measure). As a consequence we only need to bound the multiplicity of eigenvalue for the matrix 
$$P_0\left(H^\omega+\sum_{i=1}^d\lambda_i P_{e_i}-z\right)^{-1}P_0$$
for $\{\lambda_i\}_{i=1}^d$ in some open rectangle of $\RR^{d}$ and $z$ in some open set of $\RR\setminus(-(\norm{H^\omega}+\max_i |\lambda_i|),\norm{H^\omega}+\max_i |\lambda_i|)$. 
The bounds are obtained in lemma \ref{lemMain} when $\{\lambda_i\}_{i=1}^d$ are chosen from certain intervals. 
The independence of $\{\lambda_i\}_{i=1}^d$ obtained in lemma \ref{multLem2} is the main reason why the results does not depend upon the strength of disorder, even though lemma \ref{lemMain} gives the bound when $\{\lambda_i\}_{i=1}^d$ is large. 
The proof of part (1) and (2) follows certain counting argument in lemma \ref{lemMain} and the lemma \ref{lem5}. The statement of the lemma \ref{lem5} can be interpreted as $0\not\in\sigma(P_0\Delta P_0)$ for certain choices of $\{l_i\}_i$.
Combining this with the inequality provided by lemma \ref{lemMain} gives the simplicity.

The proof of the theorem is divided into three sections. In second section, criterion for bounding the multiplicity are given, and they do not depend upon the particular form of the operator \eqref{MainOperator}. 
Third section contains the lemma \ref{lemMain} which is highly specialised for the operator. The proof of the theorem \ref{mainthm} is also present in this section. 
Most of the specialised computations related to the operator \eqref{MainOperator} can be found in the fourth section. 

In the description of the model \eqref{MainOperator} we can let $\{\omega_n\}_n$ to be independent real bounded random variables following absolutely continuous distribution. 
The fact that $\{\omega_n\}_n$ are identically distributed is not used in the proof, only independence and absolute continuity of the distribution for $\{\omega_n\}_n$ are used. 
Though we do need $H^\omega$ to be bounded for lemma \ref{multLem2} to work.
So if we assume that the random variables $\{\omega_n\}_n$ are independent and follows absolutely continuous distribution, we need to assume that the random variable $\sup_{n\in\ZZ^d}|\omega_n|$ is bounded.

\noindent{\bf Acknowledgement:}  The author thank Krishna Maddaly for valuable discussion and suggestions.
\section{Criterion for bounding Multiplicity}
Following lemma provides an easy way to bound the multiplicity of singular spectrum for the operator \eqref{MainOperator}. The only result used in this lemma which depends on the particular structure of the operator $H^\omega$ is  proved in the lemma  \ref{lemInvertCond}.
\begin{lemma}\label{multLem1}
The maximum multiplicity of singular spectrum of $H^\omega$ is bounded by essential supremum of
$$f(x):=\text{maximum eigenvalue multiplicity of the matrix }P_n(H^\omega-x-\iota 0)^{-1}P_n$$
with respect to Lebesgue measure.
\end{lemma}
\begin{proof}
First few notations are needed, for any $n\in\ZZ^d$ set
\begin{equation}\label{condEq1}
 \Hi^\omega_n=\overline{\{f(H^\omega)\phi:\phi\in P_n\ell^2(\ZZ^d),f\in C_c^\infty(\RR)\}},
\end{equation}
the closed $H^\omega$-invariant subspace of $\ell^2(\ZZ^d)$ containing $\ell^2(\Lambda(n))$. 
Set $E^\omega$ and $E^\omega_{sing}$ to be the spectral measure for the operator $H^\omega$ and orthogonal projection onto the singular part of spectrum for the operator $H^\omega$ respectively. 
As a consequence of the Spectral theorem (see \cite[Theorem A.3]{AM1}) we have
$$(\Hi^\omega_n,H^\omega)\cong (L^2(\RR,P_n E^\omega(\cdot)P_n,P_n\ell^2(\ZZ^d)),Id)$$
where the operator $Id$ is multiplication by identity. 
Since the measure $P_n E^\omega(\cdot)P_n$ is absolutely continuous with respect to the trace measure $\sigma^\omega_n(\cdot)=tr(P_n E^\omega(\cdot)P_n)$, there exists a matrix $M^\omega_n\in L^1(\RR,\sigma^\omega_n, M_{rank(P_n)}(\CC))$ such that 
$$P_nE^\omega(dx)P_n=M_n^\omega(x)\sigma^\omega_n(dx)$$
for almost all $x$ with respect to $\sigma^\omega_n$. Finally since $P_n E^\omega(\cdot)P_n$ is non-negative, we have $M_n^\omega(x)\geq 0$ for almost all $x$ w.r.t $\sigma^\omega_n$.

We will use theorem \ref{singSubSpThm} to obtain 
\begin{equation}\label{singSpecEq1}
 E_{sing}^\omega\ell^2(\ZZ^d)=E_{sing}^\omega\Hi^\omega_n\qquad\forall n\in\ZZ^d
\end{equation}
almost surely. But for using theorem \ref{singSubSpThm} we need to show
\begin{equation}\label{condEq2}
 \mathbb{P}(\omega: Q^\omega_n P_m\text{ has same rank as }P_m)=1
\end{equation}
for any $n,m\in\ZZ^d$, here $Q^\omega_n$ is the canonical projection from $\ell^2(\ZZ^d)$ to $\Hi^\omega_n$. This is proved in lemma \ref{lemInvertCond}.

The Borel transform of the measure $P_n E^\omega(\cdot)P_n$ is $P_n(H^\omega-z)^{-1}P_n$,  and using Poltoratskii's theorem we have
\begin{equation}\label{polEq1}
 \lim_{\epsilon\downarrow 0}\frac{1}{tr(P_n (H^\omega-x-\iota\epsilon)^{-1}P_n)}P_n (H^\omega-x-\iota\epsilon)^{-1}P_n= M^\omega_n(x)
\end{equation}
for almost all $x$ w.r.t $\sigma^\omega_{n,sing}$, where the measure $\sigma^\omega_{n,sing}$ denotes the singular part of the measure $\sigma^\omega_n$ with respect to Lebesgue measure. Since
$$E_{sing}^\omega\Hi^\omega_n\cong L^2(\RR, M_n^\omega(x)d\sigma^\omega_{n,sing}(x),P_n\ell^2(\ZZ^d)),$$
to get the multiplicity of singular part all one needs to do is find essential supremum of the function $f(x)=rank(M_n^\omega(x))$ with respect to the measure $\sigma^\omega_{n,sing}$. As a consequence of resolvent equation (see \cite[equation (3.4)]{AM1}) we have
\begin{align}\label{rangeSetEq1}
range\left(\lim_{\epsilon\downarrow 0}\frac{1}{tr(P_n(H^\omega+\lambda P_n-x-\iota\epsilon)^{-1}P_n)}P_n(H^\omega+\lambda P_n-x-\iota\epsilon)^{-1}P_n\right)\nonumber\\
\subseteq ker(I+\lambda\lim_{\epsilon\downarrow 0} P_n(H^\omega-x-\iota \epsilon)^{-1}P_n),
\end{align}
for all $x$ such that $\lim_{\epsilon\downarrow 0}P_n(H^\omega-x-\iota \epsilon)^{-1}P_n$ exists and 
$$\lim_{\epsilon\downarrow 0}\frac{1}{tr(P_n(H^\omega+\lambda P_n-x-\iota\epsilon)^{-1}P_n)}=0.$$
We will use $P_n(H^\omega-x-\iota 0)^{-1}P_n$ to denote $\lim_{\epsilon\downarrow 0}P_n(H^\omega-x-\iota \epsilon)^{-1}P_n$. For $\omega\in\Omega$ fixed, the set 
$$\{x\in\RR:P_n(H^\omega-x-\iota 0)^{-1}P_n\text{ exists}\}$$ 
has full Lebesgue measure. Let $S\subseteq \{x\in\RR:P_n(H^\omega-x-\iota 0)^{-1}P_n\text{ exists}\}$ be a set of full Lebesgue measure, then using spectral averaging result \cite[Corollary 4.2]{CH1} one has $\sigma^{\tilde{\omega}^\lambda}_n(\RR\setminus S)=0$ for almost all $\lambda$ w.r.t. Lebesgue measure.
Here we use the notation $\tilde{\omega}^\lambda_k=\omega_k$ for all $k\neq n$, and $\tilde{\omega}^\lambda_n=\omega_n+\lambda$. So the essential support of $\sigma^{\tilde{\omega}^\lambda}_{n,sing}$ is 
$$S^\omega_\lambda=\{x\in S: \lim_{\epsilon\downarrow 0}(tr(H^\omega+\lambda P_n-x-\iota \epsilon))^{-1}=0\}$$
for almost all $\lambda$ w.r.t Lebesgue measure.
As a consequence of \eqref{polEq1} and \eqref{rangeSetEq1}, the multiplicity of singular spectrum for the operator $H^\omega+\lambda P_n$ is upper bounded by
$$ess\text{-} sup\{dim(ker(\lambda^{-1}I+ P_n(H^\omega-x-\iota 0)^{-1}P_n)):\text{for all $x$ w.r.t }\sigma^{\tilde{\omega}^\lambda}_{n,sing}\}$$
for almost all $\lambda$ w.r.t. Lebesgue measure. 
Since we can leave any zero Lebesgue measure set, giving an upper bound on the multiplicity of eigenvalues for $P_n(H^\omega-x-\iota 0)^{-1}P_n$ for $x\in S$ is enough. Hence proving the claim.

\end{proof}

As a consequence of above lemma, it is enough to bound the multiplicity of $P_0(H^\omega-x-\iota 0)^{-1}P_0$ on a set of full Lebesgue measure. Following lemma is a modification of \cite[Lemma A.2]{NNS} and provides a way of doing this. 
Part of the claim is that the maximum multiplicity of eigenvalues of $P_n(H^\omega-z)^{-1}P_n$ on its domain of definition is independent of $z$. Other part is that the maximum eigenvalue multiplicity is independent of single perturbation.
\begin{lemma}\label{multLem2}
Let  $J\subset\RR$ be an open interval and $U\subset\RR\setminus([-\norm{H^\omega},\norm{H^\omega}]+J)$ be a positive Lebesgue measure set such that the maximum multiplicity of eigenvalues of the matrix $P_0(H^\omega+\lambda P_n-z)^{-1}P_0$ is at most $k$ for $(z,\lambda)\in U\times J$. 
Then the set
\begin{align*}
Sim_\lambda:=\{E\in\RR: \lim_{\epsilon\downarrow 0} P_0(H^\omega+\lambda P_n-E-\iota \epsilon)^{-1}P_0\text{ exists and maximum}\\
 \text{multiplicity of eigenvalue is }k\}
\end{align*}
has full Lebesgue measure for almost all $\lambda$ w.r.t Lebesgue measure.
\end{lemma}
\begin{proof}
We will use the notation $G^\lambda_{p,q}(z)=P_p(H^\omega+\lambda P_n-z)^{-1}P_q$, then using the resolvent identity we have
$$G_{0,0}^\lambda(z)=G_{0,0}^0(z)-\lambda G^0_{0,n}(z)(I+\lambda G_{n,n}^0(z))^{-1}G_{n,0}^{0}(z).$$
Viewing $G^\lambda_{p,q}(z)$ as matrix over the standard basis of $P_m\ell^2(\ZZ^d)$ we define the polynomial
\begin{align*}
\tilde{g}_{\lambda,z}(x)&=\det(G_{0,0}^\lambda(z)-xI)\\
&=\det\left(G_{0,0}^0(z)-\lambda G^0_{0,n}(z)(I+\lambda G_{n,n}^0(z))^{-1}G_{n,0}^{0}(z)-xI\right)\\
&=\frac{p_l(z,\lambda)x^l+p_{l-1}(z,\lambda)x^{l-1}+\cdots+p_0(z,\lambda)}{\det(I+\lambda G^0_{n,n}(z))}
\end{align*}
where $l=rank(P_0)$. Here the coefficients $\{p_i(z,\lambda)\}_{i=0}^l$ are polynomials of matrix coefficients of $\{G^0_{p,q}(z)\}_{p,q=0,n}$ and $\lambda$.
Since denominator is independent of $x$ we only need to focus on
$$g_{\lambda,z}(x)=p_l(z,\lambda)x^l+p_{l-1}(z,\lambda)x^{l-1}+\cdots+p_0(z,\lambda).$$
The function
$$\mathcal{F}_{\lambda,z}(x)=gcd\left(g_{\lambda,z}(x),\frac{dg_{\lambda,z}}{dx}(x),\cdots,\frac{d^kg_{\lambda,z}}{dx^k}(x)\right)$$
is a constant function with respect to $x$ if the maximum multiplicity of eigenvalues of $G_{0,0}^\lambda(z)$ is $k$. Using 
$$gcd(p_1(x),\cdots,p_m(x))=gcd(p_1(x),\cdots,p_{n-2}(x),gcd(p_{n-1}(x),p_{n}(x)))$$
and Euclid's algorithm for polynomials, we have
$$\mathcal{F}_{\lambda,z}(x)=q_0(\lambda,z)+q_1(\lambda,z)x+\cdots+q_m(\lambda,z)x^m,$$
where $q_i$ are rational polynomials of $\{p_j(z,\lambda)\}_j$. We only need to focus on the numerators of $\{q_i\}_{i\geq 1}$ which will be denoted by $\{\tilde{q}_i\}_{i\geq 1}$. It is clear that $\tilde{q}_i$ are polynomial of matrix coefficients of $\{G^0_{p,q}(z)\}_{p,q=0,n}$ and $\lambda$,  hence are well defined on $(z,\lambda)\in(\CC\setminus\sigma(H^\omega))\times \RR$.
So we have $\tilde{q}_i(z,\lambda)=\sum_{j=0}^{m_i} a^i_j(z) \lambda^j$ for each $i$. The functions $\{a^i_j\}$ are holomorphic on $\CC\setminus\sigma(H^\omega)$.

Now using the fact that the maximum multiplicity of eigenvalues of $G_{0,0}^\lambda(z)$ is $k$ for $(z,\lambda)\in U\times J$, we have
$$a^i_j(z)=0\qquad\forall z\in U, i>0,1\leq j\leq m_i.$$
Since $U$ has positive Lebesgue measure, this implies (because of properties of zero set of holomorphic functions, see \cite{BR1}) that the holomorphic functions $a^i_j$ are identically zero on the domain. So the function $\mathcal{F}_{\lambda,z}(x)$ is constant with respect to $x$ for $(z,\lambda)\in (\CC\setminus\sigma(H^\omega))\times\RR$. So we get that the maximum multiplicity of the matrix $G_{0,0}^\lambda(z)$ is $k$ for $(z,\lambda)\in (\CC\setminus\sigma(H^\omega))\times\RR$, hence the conclusion follows.

\end{proof}
Using this lemma inductively, all we need to do is bound the multiplicity of eigenvalues of the matrix 
$$G^{\omega,\lambda}_{0,0}(z)=P_0\left(H^\omega+\sum_{i=1}^m\lambda_i P_{n_i}-z\right)^{-1}P_0$$
for $\lambda_i\in I_i$, $n_i\in\ZZ^d$ for $1\leq i\leq m$ for some $m\in\NN$ and $z\in U$ a positive Lebesgue measure set. Here the choice of the intervals $I_i\subset\RR$ and $U\subset\RR$ are up to us. The reason to confine the set $U$ in $\RR$ is because otherwise we loose the normality of the matrix $G^{\omega,\lambda}_{0,0}(z)$.
\section{Proof of main result}
Using the lemma \ref{multLem2} inductively it is enough to bound the multiplicity of eigenvalues for the matrix
$$G^{\omega,\lambda}_{0,0}(z)=P_0\left(H^\omega+\sum_{i=1}^d \lambda_i P_{e_i}-z\right)^{-1}P_0$$
for $(\lambda_1,\cdots,\lambda_d)$ in some open rectangle of $\prod_{i=1}^d I_i\in\RR^{d}$ and $z$ in some interval $U\subset \RR$ far enough.

Setting
$$\tilde{H}^{\omega,\lambda}=P_0\Delta P_0+(I-P_0)\Delta(I-P_0)+\sum_{n}\omega_n P_n+\sum_{i=1}^d\lambda_i P_{e_i},$$
then using resolvent equation we have
$$G^{\omega,\lambda}_{0,0}(z)=\left(P_0\Delta P_0+\omega_0 P_0-zP_0-P_0\Delta(I-P_0)(\tilde{H}^{\omega,\lambda}-z)^{-1}(I-P_0)\Delta P_0 \right)^{-1}.$$
The maximum multiplicity of $G^{\omega,\lambda}_{0,0}(z)$ is same as maximum multiplicity of
$$H^{\omega,\lambda}_z=P_0\Delta P_0-P_0\Delta(I-P_0)(\tilde{H}^{\omega,\lambda}-z)^{-1}(I-P_0)\Delta P_0.$$
Choosing $z\in \RR$ will make sure that above matrix is self adjoint and so diagonalizable. So any discrepancy between geometric multiplicity and algebraic multiplicity of eigenvalues does not arises. 
\begin{lemma}\label{lemMain}
There exists an open rectangle $\prod_{i=1}^d I_i\subset\RR^{d}$ such that the eigenvalues of the matrix 
$$r^2P_0\Delta P_0-r^2P_0\Delta(I-P_0)(\tilde{H}^{\omega,\lambda}-r)^{-1}(I-P_0)\Delta P_0$$
for $\lambda_i\in I_i$ for all $i$ and $r\gg\mathcal{R}(\norm{H^\omega},\max_{i}\sup_{\lambda_i\in I_i}|\lambda_i|,d,\{l_i\}_i)$ are of the form:
\begin{align}\label{lemMainEigRes}
E_{n_1,\cdots,n_d}&=2r^2\sum_{i=1}^d \cos\frac{\pi n_i}{l_i+1}+4r\sum_{i=1}^d \frac{1}{l_i+1}\sin^2\frac{\pi n_i}{l_i+1}\nonumber\\
&+2 \sum_{i=1}^d \frac{\omega_{e_i}+\omega_{-e_i}+\lambda_{i}}{l_i+1}\sin^2\frac{\pi n_i}{l_i+1}+D_{n_1,\cdots,n_d}(\omega,\lambda,r)
\end{align}
for $n_i\in\{1,\cdots,l_i\}$ for all $i$. We have the bound 
$$|D_{n_1,\cdots,n_d}(\omega,\lambda,r)|\leq 20d^3\max_i(l_i+1)^3.$$
Finally we have
$$|E_{n_1,\cdots,n_d}-E_{m_1,\cdots,m_d}|>10d^3\max_i(l_i+1)^3$$
whenever there exists $i$ such that $m_i\not\in\{n_i,l_i+1-n_i\}$.
\end{lemma}
\begin{proof}
The choices of the intervals will be done later. First observe that for $r\gg\norm{H^\omega}+2d\max_{n}|\lambda_n|$, we have
\begin{align*}
&H_r^{\omega,\lambda}=P_0\Delta P_0-P_0\Delta(I-P_0)(\tilde{H}^{\omega,\lambda}- r)^{-1}(I-P_0)\Delta P_0\\
&\qquad=P_0\Delta P_0+\frac{1}{ r}P_0\Delta(I-P_0)\Delta P_0+\frac{1}{r^2}P_0\Delta (I-P_0)\tilde{H}^{\omega,\lambda}(I-P_0)\Delta P_0+O\left(\frac{1}{r^3}\right)\\
&\qquad=P_0\Delta P_0+\frac{1}{ r}P_0\Delta(I-P_0)\Delta P_0+\frac{1}{r^2}P_0\Delta (I-P_0)\Delta(I-P_0)\Delta P_0\\
&\qquad\qquad+\frac{1}{r^2} \sum_{\norm{n}_1=1}\omega_n P_0\Delta P_n\Delta P_0+\frac{1}{r^2} \sum_{i=1}^d \lambda_i P_0\Delta P_{e_i}\Delta P_0+O\left(\frac{1}{r^3}\right).
\end{align*}
Notice that
\begin{align*}
 P_0\Delta(I-P_0)\Delta P_0&=P_0\Delta\left(\sum_{n\in\ZZ^d\setminus\{0\}}P_n\right)\Delta P_0=\sum_{\norm{n}_1=1}P_0\Delta P_n\Delta P_0
\end{align*}
because the $P_0\Delta P_n\Delta P_0$ can be non-zero if and only if the distance between the $\Lambda(n)$ and $\Lambda(0)$ is at most $1$.\\
\begin{minipage}{0.6\textwidth}
Next observe that (using Dirac notation)
\begin{align*}
&P_0\Delta P_{e_i} \Delta P_0\\
&=\sum_{\substack{p,q\in\Lambda_l(0)\\s\in\Lambda_l(e_i) }}\proj{\delta_p}\Delta\proj{\delta_s}\Delta\proj{\delta_q}\\
&=\sum_{\substack{p,q\in\Lambda_l(0)\\ s\in\Lambda_l(e_i)\\|p-s|=1\& |q-s|=1}} |\delta_p\rangle\langle\delta_q|.
\end{align*}
\end{minipage}
\begin{minipage}{0.35\textwidth}
\centering
\includegraphics[width=2.4in,keepaspectratio]{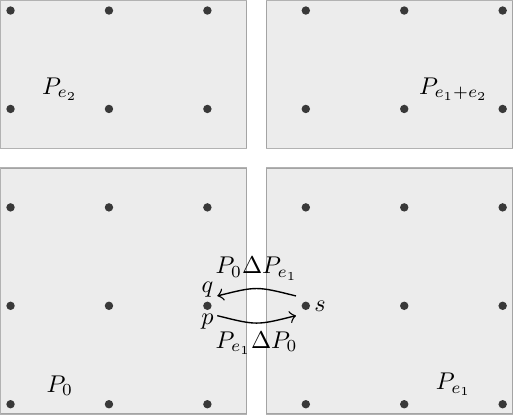}
\end{minipage}

\noindent For $p,q\in\Lambda_l(0)$ and $s\in\Lambda_l(e_i)$ such that $\norm{p-s}_1=1$ and $\norm{q-s}_1=1$, only possibility is if $s=p+e_i$ and $s=q+e_i$. So this implies $p=q$ and $s$ is unique. 
Hence the matrix $P_0\Delta P_{e_i}\Delta P_0$ is projection on $\{x\in\Lambda(0):x_i=l_i\}$, and similarly $P_0\Delta P_{-e_i}\Delta P_0$ is projection on $\{x\in\Lambda(0):x_i=1\}$.
These observations give us
\begin{align*}
&A_r^{\omega,\lambda}=r^2 P_0\Delta P_0+rP_0\Delta(I-P_0)\Delta P_0+\sum_{\norm{n}_1=1}\omega_n P_0\Delta P_n\Delta P_0+\sum_{i=1}^d \lambda_i P_0\Delta P_{e_i}\Delta P_0\\
&=\sum_{i=1}^d I^{i-1}\otimes (r^2\Delta_L+(\omega_{-e_i}+ r)\proj{\delta_1}+(\omega_{e_i}+\lambda_{i}+ r)\proj{\delta_l})\otimes I^{d-i}.
\end{align*}
Hence for $r$ large enough, using the lemma \ref{lemEigAprx} the eigenvalues of above matrix are of form
\begin{align*}
\tilde{E}_{n_1,\cdots,n_d}&=2r^2\sum_{i=1}^d \cos\frac{\pi n_i}{l_i+1}+4r\sum_{i=1}^d \frac{1}{l_i+1}\sin^2\frac{\pi n_i}{l_i+1}\\
&+2 \sum_{i=1}^d \frac{\omega_{e_i}+\omega_{-e_i}+\lambda_{i}}{l_i+1}\sin^2\frac{\pi n_i}{l_i+1}+\sum_{i=1}^d C_{l_i,n_i}+O\left(\frac{1}{r}\right)
\end{align*}
where
$$C_{l,n}=\frac{2}{(l+1)^2}\sin^2\frac{\pi n}{l+1}\sum_{\substack{m\neq n\\ m\equiv n mod 2}}\frac{\sin^2\frac{\pi m}{l+1}}{\cos\frac{\pi m}{l+1}-\cos\frac{\pi n}{l+1}}.$$
Now using the lemma \ref{lem4} with $\frac{1}{200d^3\max_i(l_i+1)^3}<\delta<\frac{1}{100d^3\max_i(l_i+1)^3}$ we have the rectangle $\prod_i I_i$ such that whenever $\sin^2\frac{\pi n_i}{l_i+1}\neq \sin^2\frac{\pi m_i}{l_i+1}$ for some $i$ we have
\begin{align}\label{lemIneqEq1}
&\left|\sum_{i=1}^d \frac{\omega_{e_i}+\omega_{-e_i}+\lambda_{i}}{l_i+1}\sin^2\frac{\pi n_i}{l_i+1}-\sum_{i=1}^d \frac{\omega_{e_i}+\omega_{-e_i}+\lambda_{i}}{l_i+1}\sin^2\frac{\pi m_i}{l_i+1}\right|\nonumber\\
&\qquad>100d^3\max_i(l_i+1)^3.
\end{align}
The condition $\sin^2\frac{\pi n_i}{l_i+1}= \sin^2\frac{\pi m_i}{l_i+1}$ can happen for $n_i,m_i\in\{1,\cdots,l_i\}$ if and only if $m_i\in\{n_i,l_i+1-n_i\}$. Hence, suppose that there exists $i$ such that $m_i\not\in\{n_i,l_i+1-n_i\}$ then using
\begin{align}\label{lemIneqEq2}
&|\tilde{E}_{n_1,\cdots,n_d}-\tilde{E}_{m_1,\cdots,m_d}|\nonumber\\
&\geq \left\{\begin{matrix} 2r^2\left|\sum_{j=1}^d\left(\cos\frac{\pi n_j}{l_j+1}-\cos\frac{\pi m_j}{l_i+1}\right)\right|+O(r) & \text{if }\sum_{j=1}^d\left(\cos\frac{\pi n_j}{l_j+1}-\cos\frac{\pi m_j}{l_i+1}\right)\neq 0\\ 
4r\left|\sum_{j=1}^d\frac{1}{l_j+1}\left(\sin^2\frac{\pi n_j}{l_j+1}-\sin^2\frac{\pi m_j}{l_i+1}\right)\right|+O(1) & \text{if }\sum_{j=1}^d\frac{1}{l_j+1}\left(\sin^2\frac{\pi n_j}{l_j+1}-\sin^2\frac{\pi m_j}{l_i+1}\right)\neq 0\\
&\text{and } \sum_{j=1}^d\left(\cos\frac{\pi n_j}{l_j+1}-\cos\frac{\pi m_j}{l_i+1}\right)=0\\
2\left|\sum_{j=1}^d\frac{\omega_{e_j}+\omega_{-e_j}+\lambda_{j}}{l_j+1}\left(\sin^2\frac{\pi n_j}{l_j+1}-\sin^2\frac{\pi m_j}{l_i+1}\right)\right| & \text{if }\sum_{j=1}^d\frac{1}{l_j+1}\left(\sin^2\frac{\pi n_j}{l_j+1}-\sin^2\frac{\pi m_j}{l_i+1}\right)= 0\\
~-\sum_{j=1}^d (C_{l_j,n_j}-C_{l_j,m_j})+O\left(\frac{1}{r}\right) &\text{and } \sum_{j=1}^d\left(\cos\frac{\pi n_j}{l_j+1}-\cos\frac{\pi m_j}{l_i+1}\right)=0\\
\end{matrix}\right.
\end{align}
Since
$$\min\left\{\left|\sum_{j=1}^d\cos\frac{\pi n_j}{l_j+1}-\cos\frac{\pi m_j}{l_i+1}\right|: \sum_{j=1}^d\cos\frac{\pi n_j}{l_j+1}\neq \sum_{j=1}^d\cos\frac{\pi m_j}{l_i+1}\right\}=\tilde{c}>0$$
and
$$\hspace{-1cm}\min\left\{\left|\sum_{j=1}^d\frac{1}{l_j+1}\left(\sin^2\frac{\pi n_j}{l_j+1}-\sin^2\frac{\pi m_j}{l_i+1}\right)\right|: \sum_{j=1}^d\frac{\sin^2\frac{\pi n_j}{l_j+1}}{l_j+1}\neq \sum_{j=1}^d\frac{\sin^2\frac{\pi m_j}{l_i+1}}{l_j+1}\right\}=\tilde{s}>0,$$
we can choose $r>\frac{200}{\min\{\tilde{c},\tilde{s}\}}\left(\frac{4d^3\max_i(l_i+1)^3}{\epsilon\delta}\right)^{d+1}$ (where $\epsilon,\delta$ are defined using lemma \ref{lem4} while getting \eqref{lemIneqEq1}) we that
\begin{align}\label{lemIneqEq3}
&\left|2r^2\sum_{j=1}^d\left(\cos\frac{\pi n_j}{l_j+1}-\cos\frac{\pi m_j}{l_i+1}\right)+4r\sum_{j=1}^d\frac{1}{l_j+1}\left(\sin^2\frac{\pi n_j}{l_j+1}-\sin^2\frac{\pi m_j}{l_i+1}\right)\right|\nonumber\\
&\qquad>200\left(\frac{4}{\epsilon\delta}\right)^{d+1}(d^{3}\max_i(l_i+1)^{3})^{d+1}
\end{align}
whenever the coefficients of $r^2$ or $r$ is non-zero. So using \eqref{lemIneqEq3} and using the bound $|C_{l,n}|<10(l+1)$ to get $\left|\sum_{j=1}^d (C_{l_j,n_j}-C_{l_j,m_j})+O\left(\frac{1}{r}\right)\right|\leq 40d\max_i(l_i+1)$, on the inequality  \eqref{lemIneqEq2} we have
\begin{equation}\label{lemMainEq1}
 |\tilde{E}_{n_1,\cdots,n_d}-\tilde{E}_{m_1,\cdots,m_d}|>50d^3\max_i(l_i+1)^3.
\end{equation}
Note that the spectrum of $A^{\omega,\lambda}_r$ is simple and the gap between the eigenvalues appear because of $r^{-1}$ order term of \eqref{lemEigAprxRes1}. Setting
$$A^{\omega,\lambda}_{r,t}=A^{\omega,\lambda}_r+t P_0\Delta (I-P_0)\Delta(I-P_0)\Delta P_0\qquad for~t\in(-1,2)$$
and using the continuity of eigenvalues as function of $t$ for $A^{\omega,\lambda}_{r,t}$, observe that corresponding to each eigenvalue $\tilde{E}_{n_1,\cdots,n_d}$ there exists an eigenvalue $E_{n_1,\cdots,n_d}^1$ of $A^{\omega,\lambda}_{r,1}$ such that 
\begin{equation}\label{lemMainEq2}
 |E_{n_1,\cdots,n_d}^{1}-\tilde{E}_{n_1,\cdots,n_d}|<(2d)^3,
\end{equation}
because $\norm{P_0\Delta (I-P_0)\Delta(I-P_0)\Delta P_0}<(2d)^3$. 
So combining above and taking care of $O(\frac{1}{r})$ term of $r^2 H^{\omega,\lambda}_r$ gives us that the eigenvalues of $r^2 H^{\omega,\lambda}_r$ are of the form
\begin{align*}
E_{n_1,\cdots,n_d}&=2r^2\sum_{i=1}^d \cos\frac{\pi n_i}{l_i+1}+4r\sum_{i=1}^d \frac{1}{l_i+1}\sin^2\frac{\pi n_i}{l_i+1}\\
&+2 \sum_{i=1}^d \frac{\omega_{e_i}+\omega_{-e_i}+\lambda_{i}}{l_i+1}\sin^2\frac{\pi n_i}{l_i+1}+D_{n_1,\cdots,n_d}(\omega,\lambda,r).
\end{align*}
Here
\begin{align*}
D_{n_1,\cdots,n_d}(\omega,\lambda,r)=\sum_{i=1}^d C_{l_i,n_i}+C_{n_1,\cdots,n_d}(\omega,\lambda,r)+O\left(\frac{1}{r}\right),
\end{align*}
and $|C_{n_1,\cdots,n_d}(\omega,\lambda,r)|\leq \norm{P_0\Delta(I-P_0)\Delta(I-P_0)\Delta P_0}\leq (2d)^3$. So we have the bound
\begin{align*}
 |D_{n_1,\cdots,n_d}(\omega,\lambda,r)|\leq 20\sum_{i=1}^d (l_i+1)^3+10d^3\leq 20d^3\max_i(l_i+1)^3.
\end{align*}
Finally using \eqref{lemMainEq2} and \eqref{lemMainEq1} together we have
$$|E_{n_1,\cdots,n_d}-E_{m_1,\cdots,m_d}|>10d^3\max_i(l_i+1)^3$$
when ever there exists $i$ such that $m_i\not\in\{n_i,l_i+1-n_i\}$.

\end{proof}
\subsection{Proof of Theorem \ref{mainthm}}
Lemma \ref{multLem1} implies that the bound on maximum multiplicity of singular spectrum for $H^\omega$ is given by the essential supremum of 
$$f(x):=\text{maxmimum eigenvalue multiplicity of the matrix }P_0(H^\omega-x-\iota 0)^{-1}P_0$$
w.r.t Lebesgue measure. 
So using lemma \ref{multLem2} inductively we have $\norm{f}_\infty=$ the maximum multiplicity of $P_0\left(H^\omega+\sum_{i=1}^d\lambda_i P_{e_i}-z\right)^{-1}P_0$ for $\{\lambda_i\}_{i=1}^d$ in some intervals and $z$ in some positive Lebesgue measure set of $\RR$. 
As seen in the beginning of the section, the maximum multiplicity of eigenvalues for 
$$P_0\left(H^\omega+\sum_{i=1}^d\lambda_i P_{e_i}-r\right)^{-1}P_0$$
is same as the maximum multiplicity of eigenvalues for 
\begin{equation}\label{pfEq2}
 r^2P_0\Delta P_0-r^2P_0\Delta(I-P_0)(\tilde{H}^{\omega,\lambda}-r)^{-1}(I-P_0)\Delta P_0
\end{equation}
whenever $r\not\in\sigma(H^\omega+\sum_{i=1}^d\lambda_i P_{e_i})$.

Lemma \ref{lemMain} provides the range for $\{\lambda_n\}_n$ and $r$ such that the eigenvalues follows the inequality
\begin{equation}\label{pfEq1}
 |E_{n_1,\cdots,n_d}-E_{m_1,\cdots,m_d}|>10d^3\max_i(l_i+1)^3
\end{equation}
whenever there exists $i$ such that $m_i\not\in\{n_i,l_i+1-n_i\}$. There are at most $2^d$ possible $(m_1,\cdots,m_d)$ such that $m_i\in\{n_i,l_i+1-n_i\}$, and so at most $2^d$ eigenvalues $E_{m_1,\cdots,m_d}$ for which the bound \eqref{pfEq1} may not hold. But if there exists an unique $i$ such that $n_i\neq m_i=l_i+1-n_i$, then (using $\cos\frac{\pi(l_i+1-n_i)}{l_i+1}=-\cos\frac{\pi n_i}{l_i+1}$)
\begin{align*}
|E_{n_1,\cdots,n_d}-E_{m_1,\cdots,m_d}|=4r^2\left|\cos\frac{\pi n_i}{l_i+1}\right|+O(r),
\end{align*}
which cannot be zero. Only possible way for $\cos\frac{\pi n_i}{l_i+1}$ to be zero is if $n_i=\frac{l_i+1}{2}$, but then $m_i=l_i+1-n_i=\frac{l_i+1}{2}=n_i$.
So multiplicity of any eigenvalue of \eqref{pfEq2} is at most $2^d-d$. 
This completes the proof of part (3) of the theorem.

For part (1) and (2) of the theorem, all we need to do is focus on the eigenvalue clusters $\{E_{m_1,\cdots,m_d}\}_{m_i\in\{n_i,l_i+1-n_i\}}$ and show that they are distinct.
\\
\\{\bf Proof of part (1):}  We need to show 
$$|E_{n_1,n_2}-E_{m_1,m_2}|>0\qquad\forall (n_1,n_2)\neq (m_1,m_2).$$
As observed in equation \eqref{pfEq1} we only need to show above for $m_i\in\{n_i,l_i+1-n_i\}$. Using $\cos\frac{\pi(l_i+1-n_i)}{l_i+1}=-\cos\frac{\pi n_i}{l_i+1}$, there are only three cases are possible:
\begin{align*}
E_{n_1,n_2}-E_{m_1,m_2}=\left\{\begin{matrix} 4r^2\left(\cos\frac{\pi n_1}{l_1+1}+\cos\frac{\pi n_2}{l_2+1}\right)+O(r) & (m_1,m_2)=(l_1+1-n_1,l_2+1-n_2) \\ 4r^2\cos\frac{\pi n_1}{l_1+1}+O(r) & (m_1,m_2)=(l_1+1-n_1,n_2) \\ 4r^2\cos\frac{\pi n_2}{l_2+1}+O(r) & (m_1,m_2)=(n_1,l_2+1-n_2) \end{matrix}\right.
\end{align*}
But for second and third cases, we already have $E_{n_1,n_2}\neq E_{m_1,m_2}$ for large enough $r$, because $\cos\frac{\pi n_i}{l_i+1}=0$ implies $n_i=\frac{l_i+1}{2}$ (because $1\leq n_i\leq l_i$), so $m_i=l_i+1-\frac{l_i+1}{2}=\frac{l_i+1}{2}=n_i$. 

For first case we have
\begin{align*}
0&=\cos\frac{\pi n_1}{l_1+1}-\cos\frac{\pi n_2}{l_2+1}\\
&=2\sin\frac{\pi}{2}\left(\frac{n_1}{l_1+1}+\frac{n_2}{l_2+1}\right)\sin\frac{\pi}{2}\left(\frac{n_1}{l_1+1}-\frac{n_2}{l_2+1}\right)\\
\Rightarrow\qquad&\frac{n_1}{l_1+1}+\frac{n_2}{l_2+1}\in\{0,1\}~~or~~\frac{n_1}{l_1+1}=\frac{n_2}{l_2+1}
\end{align*}
First case cannot occur because $1\leq n_i\leq l_i$ and second case cannot occur because $gcd(l_1+1,l_2+1)=1$.

So \eqref{pfEq2} has simple eigenvalues for $r$ and $\{\lambda_n\}_{\norm{n}_1=1}$ chosen as per lemma \ref{lemMain}. Hence the simplicity of singular spectrum for $H^\omega$.
\\
\\{\bf Proof of part (2):}  Similar to part (1), we only have to show
$$|E_{n_1,\cdots,n_d}-E_{m_1,\cdots,m_d}|>0\qquad\forall (n_1,\cdots,n_d)\neq (m_1,\cdots,m_d).$$
Using equation \eqref{pfEq1} we only need to show above only when $m_i\in\{n_i,l_i+1-n_i\}$. 
Let $S=\{i:n_i\neq m_i\}$ then for $(m_1,\cdots,m_d)\neq (n_1,\cdots,n_d)$ observe
\begin{align*}
E_{n_1,\cdots,n_d}-E_{m_1,\cdots,m_d}=4r^2\sum_{i\in S}\cos\frac{\pi n_i}{l_i+1}+O(r).
\end{align*}
Using the lemma \ref{lem5} on the indexing set $S$ we get
$$\sum_{i\in S} \cos\frac{\pi n_i}{l_i+1}\neq 0$$
for any $n_i\in\{1,\cdots,l_i\}$ for all $i$. So \eqref{pfEq2} has simple eigenvalues for $r$ and $\{\lambda_n\}_{\norm{n}_1=1}$ chosen as per lemma \ref{lemMain}. Hence the simplicity of singular spectrum for $H^\omega$.
\subsection{Proof of Corollary \ref{corMainThm}}
Following the argument of the proof of theorem \ref{mainthm} from above, we have to show the bounds of eigenvalue multiplicity for
$$r^2H_r^{\omega,\lambda}=r^2P_0\Delta P_0-r^2P_0\Delta(I-P_0)(\tilde{H}^{\omega,\lambda}-r)^{-1}(I-P_0)\Delta P_0.$$
In the proof of the lemma \ref{lemMain} we have
\begin{align*}
& r^2H_r^{\omega,\lambda}=r^2P_0\Delta P_0+rP_0\Delta(I-P_0)\Delta P_0+P_0\Delta (I-P_0)\Delta(I-P_0)\Delta P_0\\
&\qquad\qquad+\sum_{\norm{n}_1=1}\omega_n P_0\Delta P_n\Delta P_0+\sum_{i=1}^d\lambda_i P_0\Delta P_{e_i}\Delta P_0+O\left(\frac{1}{r}\right).
\end{align*}
and 
$$P_0\Delta(I-P_0)\Delta P_0=\sum_{\norm{n}_1=1}P_0\Delta P_n\Delta P_0$$
where $P_0\Delta P_{e_i}\Delta P_0$ are projection on $\{x\in \Lambda(0): x_i=l_i\}$. Since $l_{s+1}=\cdots=l_d=1$ we have 
\begin{align*}
&A_r^{\omega,\lambda}=r^2 P_0\Delta P_0+rP_0\Delta(I-P_0)\Delta P_0+\sum_{\norm{n}_1=1}\omega_nP_0\Delta P_n\Delta P_0+\sum_{i=1}^d \lambda_i P_0\Delta P_{e_i}\Delta P_0\\
&=\sum_{i=1}^s I^{i-1}\otimes (r^2\Delta_L+(\omega_{-e_i}+ r)\proj{\delta_1}+(\omega_{e_i}+\lambda_{i}+ r)\proj{\delta_l})\otimes I^{d-i}\\
&\qquad+ \left(\sum_{i=s+1}^d \omega_{e_i}+\omega_{-e_i}+\lambda_{i}+ 2r\right)I^{d}.
\end{align*}
So the eigenvalues of $r^2H^{\omega,\lambda}_r$ obtained through following the steps of lemma \ref{lemMain} will be of the form
\begin{align*}
E_{n_1,\cdots,n_s}&=2r^2\sum_{i=1}^s \cos\frac{\pi n_i}{l_i+1}+4r\sum_{i=1}^s \frac{1}{l_i+1}\sin^2\frac{\pi n_i}{l_i+1}\\
&\qquad+2 \sum_{i=1}^s \frac{\omega_{e_i}+\omega_{-e_i}+\lambda_{i}}{l_i+1}\sin^2\frac{\pi n_i}{l_i+1}\\
&\qquad+\sum_{i=s+1}^d (\omega_{e_i}+\omega_{-e_i}+\lambda_{i}+ 2r)+D_{n_1,\cdots,n_s}(\omega,\lambda,r).
\end{align*}
where we have the bound $|D_{n_1,\cdots,n_s}(\omega,\lambda,r)|<20d^3\max_i (l_i+1)^3$. Hence the conclusions follows by imitating the steps of proof for theorem \ref{mainthm}.

\section{Important Results}
The results given in this section are used in second and third sections. Most of the content here are either computational or technical in nature.

Following lemma is used to get \eqref{lemMainEq1} where $x_{i,j}=\frac{1}{l_i+1}\sin^2\frac{\pi j}{l_i+1}$. In lemma \ref{lemMain}, we need to get hold of certain intervals for $\{\omega_{e_i}+\omega_{-e_i}+\lambda_{i}\}_{i=1}^d$ which is given by the range of $a_i$ in the following lemma.
\begin{lemma}\label{lem4}
Let $S_i:=\{x_{i,1},\cdots,x_{i,N_i}\}\subset (0,1)$ be given for $1\leq i\leq d$, define 
$$0<\epsilon=\min_i\{\min_j x_{i,j},\min_{i\neq j}|x_{j,k}-x_{i,k}|\},$$
and $0<\delta<\min\{\frac{1}{2},\frac{1}{1+\epsilon}\}$, then for $\{a_i\}_{i=1}^d$ such that $\frac{1}{2} \left( \frac{2}{\epsilon\delta} \right)^{i}<a_i<\left(\frac{2}{\epsilon\delta}\right)^i$ for all $i$, the sum
$$F_\pi=\sum_{i=1}^d a_i x_{i,\pi(i)}$$
where $\pi\in\prod_{i=1}^d\{1,\cdots,N_i\}$ are unique for each $\pi$. Finally for $\pi\neq\psi$
\begin{equation}\label{lem4eq1}
|F_\pi-F_\psi|>\frac{1}{\delta}.
\end{equation}
\end{lemma}
\begin{proof}
All we have to show is $F_\pi\neq F_\psi$ if $\pi,\psi\in\prod_{i=1}^d\{1,\cdots,N_i\}$ such that $\pi(i)\neq \psi(i)$ for some $i$.

Let $k=\max\{i:\pi(i)\neq\psi(i)\}$, then we have
\begin{align*}
F_\pi-F_\psi&=\sum_{i=1}^d a_i(x_{i,\pi(i)}-x_{i,\psi(i)})\\
&=a_k(x_{k,\pi(k)}-x_{k,\psi(k)})+\sum_{i=1}^{k-1} a_i(x_{i,\pi(i)}-x_{i,\psi(i)})\\
|F_\pi-F_\psi|&\geq a_k|x_{k,\pi(k)}-x_{k,\psi(k)}|-a_k\sum_{i=1}^{k-1}\frac{a_i}{a_k}|x_{i,\pi(i)}-x_{i,\psi(i)}|\\
&\geq a_k|x_{k,\pi(k)}-x_{k,\psi(k)}|-a_k\sum_{i=1}^{k-1}(\epsilon\delta)^{k-i}\\
&\geq a_k|x_{k,\pi(k)}-x_{k,\psi(k)}|-a_k\frac{\epsilon\delta}{1-\epsilon\delta}\\
&\geq a_k\left(\epsilon-\frac{\epsilon\delta}{1-\epsilon\delta}\right)=a_k\epsilon\frac{1-(1+\epsilon)\delta}{1-\epsilon\delta}>0
\end{align*}
This gives us the estimate:
\begin{equation*}
 |F_\pi-F_\psi|\geq \frac{1-(1+\epsilon)\delta}{1-\epsilon\delta}\frac{1}{\epsilon^{k-1}\delta^k}>\frac{1}{\delta}
\end{equation*}
Completing the proof.

\end{proof}
To prove part (2) of theorem \ref{mainthm}, we needed to show $\sum_{i\in S}\cos\frac{\pi n_i}{l_i+1}\neq 0$ for $S\subset\{1,\cdots,d\}$. This is done in the following by using properties of field extension for roots of unity.
\begin{lemma}\label{lem5}
Let $p_1,\cdots,p_d\in \NN\setminus 2\NN\cup3\NN\cup\{1\}$ be such that $gcd(p_i,p_j)=1$ for all $i\neq j$.
Then
$$\sum_{i=1}^d \cos\frac{\pi n_i}{p_i}\neq 0$$
for any choice of $n_i\in\{1,\cdots,p_i-1\}$ for all $i$.
\end{lemma}
\begin{proof}
The proof uses properties of algebraic extensions over $\mathbb{Q}$. We will use $\mathbb{Q}(\alpha)$ to denote the minimal field extension of $\mathbb{Q}$ which contains $\alpha$. First let us prove the following statement:
\\
\\{\bf Claim:} Let $p,q\in\NN\setminus 2\NN\cup3\NN\cup\{1\}$ with $gcd(p,q)=1$, then $\cos\frac{\pi k}{p}\not\in\mathbb{Q}(e^{\iota\frac{\pi}{q}})$ for any $1\leq k<p$.
\\{\bf Proof:} First notice that $\mathbb{Q}(e^{\iota\frac{\pi}{p}},e^{\iota\frac{\pi}{q}})=\mathbb{Q}(e^{\iota\frac{\pi}{pq}})$. The inclusion $\mathbb{Q}(e^{\iota\frac{\pi}{pq}})\subseteq \mathbb{Q}(e^{\iota\frac{\pi}{p}},e^{\iota\frac{\pi}{q}})$ is always there. For other inclusion notice that any element $a\in \mathbb{Q}(e^{\iota\frac{\pi}{p}},e^{\iota\frac{\pi}{q}})$ is of the form $\sum_{i,j}a_{i,j}e^{\iota\pi\left(\frac{i}{p}+\frac{j}{q}\right)}$ where $a_{i,j}\in\mathbb{Q}$, so $a$ always belongs to $\mathbb{Q}(e^{\iota\frac{\pi}{pq}})$. Hence we have the following field extensions
\begin{align*}
\xymatrix{
& \mathbb{Q}(e^{\iota\frac{\pi}{pq}})\ar@{-}[ld]\ar@{-}[rd]\ar@{-}[dd]_{\phi(pq)}&\\
\mathbb{Q}(e^{\iota\frac{\pi}{p}})\ar@{-}[rd]_{\phi(p)} & & \mathbb{Q}(e^{\iota\frac{\pi}{q}})\ar@{-}[ld]^{\phi(q)}\\
&\mathbb{Q}&
}
\end{align*}
where $\phi$ is the Euler $\phi$-function which has the property $\phi(pq)=\phi(p)\phi(q)$ for $gcd(p,q)=1$. So we have
$$[\mathbb{Q}(e^{\iota\frac{\pi}{pq}}):\mathbb{Q}]=[\mathbb{Q}(e^{\iota\frac{\pi}{q}}):\mathbb{Q}][\mathbb{Q}(e^{\iota\frac{\pi}{p}}):\mathbb{Q}]$$
which implies $\mathbb{Q}(e^{\iota\frac{\pi}{p}})\cap \mathbb{Q}(e^{\iota\frac{\pi}{q}})=\mathbb{Q}$. 
This is because for $\alpha\in\mathbb{Q}(e^{\iota\frac{\pi}{p}})\cap \mathbb{Q}(e^{\iota\frac{\pi}{q}})\setminus\mathbb{Q}$ we have $\alpha=\sum_{i=0}^{\phi(q)-1} b_i e^{\iota\frac{\pi i}{q}}$ where $b_i\in\mathbb{Q}$ for all $i$. 
The polynomial $\tilde{m}(x)=\sum_{i=0}^{\phi(q)-1}b_i x^i-\alpha\in\mathbb{Q}(e^{\iota\frac{\pi}{p}})[x]$ and $e^{\iota\frac{\pi}{q}}$ is a root. 
But above diagram implies $[\mathbb{Q}(e^{\iota\frac{\pi}{pq}}):\mathbb{Q}(e^{\iota\frac{\pi}{p}})]=\phi(q)$, so the degree of the minimal polynomial $m(x)$ of $e^{\iota\frac{\pi}{q}}$ over $\mathbb{Q}(e^{\iota\frac{\pi}{p}})$ is $\phi(q)$.
 Since $m(x)|\tilde{m}(x)$ degree of $m$ has to be less then degree of $\tilde{m}$. Hence the contradiction.

So only way $\cos\frac{\pi k}{p}\in\mathbb{Q}(e^{\iota\frac{\pi}{q}})$ is if $\cos\frac{\pi k}{p}\in\mathbb{Q}$. Let $\frac{\tilde{k}}{\tilde{p}}=\frac{k}{p}$ such that $gcd(\tilde{k},\tilde{p})=1$, so there exists $\tilde{n}$ such that $\tilde{n}\tilde{k}\equiv 1 mod \tilde{p}$. Using this we have
\begin{align*}
\cos\frac{\pi}{\tilde{p}}&=\cos\frac{\pi \tilde{k}\tilde{n}}{\tilde{p}}=\Re\left(e^{\iota\pi\frac{\tilde{k}}{\tilde{p}}}\right)^{\tilde{n}}=\Re\left(\cos\frac{\pi \tilde{k}}{\tilde{p}}+\iota\sin\frac{\pi \tilde{k}}{\tilde{p}}\right)^{\tilde{n}}\\
&=\sum_{r=0}^{\lfloor\frac{\tilde{n}}{2}\rfloor}(-1)^r\matx{\tilde{n}\\ 2r}\left(\sin\frac{\pi \tilde{k}}{\tilde{p}}\right)^{2r}\left(\cos\frac{\pi \tilde{k}}{\tilde{p}}\right)^{\tilde{n}-2r}\\
&=\sum_{r=0}^{\lfloor\frac{\tilde{n}}{2}\rfloor}(-1)^r\matx{\tilde{n}\\ 2r}\left(1-\cos^2\frac{\pi \tilde{k}}{\tilde{p}}\right)^{r}\left(\cos\frac{\pi \tilde{k}}{\tilde{p}}\right)^{\tilde{n}-2r}.
\end{align*}
So if $\cos\frac{\pi k}{p}\in\mathbb{Q}$ then $\cos\frac{\pi}{\tilde{p}}\in\mathbb{Q}$. Hence assume $\cos\frac{\pi}{\tilde{p}}=\alpha\in\mathbb{Q}$, then we have
\begin{align*}
\cos\frac{\pi}{\tilde{p}}=\alpha~\Leftrightarrow~e^{\iota\frac{\pi}{\tilde{p}}}+e^{-\iota\frac{\pi}{\tilde{p}}}=2\alpha~\Leftrightarrow~e^{2\iota\frac{\pi}{\tilde{p}}}-2\alpha e^{\iota\frac{\pi}{\tilde{p}}}+1=0.
\end{align*}
So $e^{\iota\frac{\pi}{\tilde{p}}}$ is root of $x^2-2\alpha x+1\in\mathbb{Q}[x]$. But degree of minimal polynomial of $e^{\iota\frac{\pi}{\tilde{p}}}$ over $\mathbb{Q}$ is given by $\phi(2\tilde{p})$, and so using the properties of minimal polynomial we get $\phi(2\tilde{p})\leq 2$. Only way this can happen is for $\tilde{p}\in\{1,2,3\}$. Since $p\not\in 2\NN\cup3\NN$, none of its factors can be $2$ or $3$, which completes the proof of the claim.
\\
\\
Using $\mathbb{Q}(e^{\iota\frac{\pi}{p}},e^{\iota\frac{\pi}{q}})=\mathbb{Q}(e^{\iota\frac{\pi}{pq}})$ whenever $gcd(p,q)=1$, inductively we have
$$\mathbb{Q}(e^{\iota\frac{\pi}{p_1}},\cdots,e^{\iota\frac{\pi}{p_{d-1}}})=\mathbb{Q}(e^{\iota\frac{\pi}{p_1\cdots p_{d-1}}}).$$
As a consequence of above claim we have 
$$\cos\frac{\pi n_d}{p_d}\not\in \mathbb{Q}(e^{\iota\frac{\pi}{p_1}},\cdots,e^{\iota\frac{\pi}{p_{d-1}}})$$
which completes the proof of the lemma.

\end{proof}
The work done in this manuscript uses the theorem \ref{singSubSpThm}. But for using the theorem we have to show \eqref{condEq2}. 
Following lemma uses the fact that the operator $H^\omega$ has the laplacian in it and that the distribution of the random variable are absolutely continuous. Then perturbation by single projection is used inductively to get the result.
\begin{lemma}\label{lemInvertCond}
For any $n,m\in\ZZ^d$, let $\Hi^\omega_n$ be defined as \eqref{condEq1} and $Q^\omega_n$ be the canonical projection form $\ell^2(\ZZ^d)$ to $\Hi^\omega_n$, then 
$$\mathbb{P}(Q^\omega_n P_m \text{ has same rank as }P_m)=1$$
\end{lemma}
\begin{proof}
We will use the notation
$$\Omega_{n,m}=\{\omega:Q^\omega_n P_m \text{ has same rank as }P_m\},$$
then the statement of the lemma is:
$$\mathbb{P}(\Omega_{n,m})=1\qquad\forall n,m\in\ZZ^d.$$
First we will show
$$\mathbb{P}(\Omega_{n,n+ke_i})=1\qquad\forall k\in\ZZ~\&~1\leq i\leq d,$$
here $\{e_i\}$ denotes the standard basis of $\ZZ^d$. The proof is through contradiction. 

Let $\omega$ is such that $rank(Q^\omega_n P_{n+ke_i})<rank(P_{n+ke_i})$, so there exists $\phi\in P_{n+ke_i}\ell^2(\ZZ^d)=\ell^2(\Lambda(n+ke_i))$ such that $P_n (H^\omega)^l \phi=0$ for all $l\geq 0$.\\
\begin{minipage}{0.6\textwidth}
Next, set $a=dist(\Lambda(n+ke_i),supp(\phi))$, let $x_0\in\Lambda(n)$ and $y_0\in supp(\phi)$ be such that $dist(x_0,y_0)=\norm{x_0-y_0}_1=a$. Note that except for $i^{th}$ entry, rest of the coordinates of $x_0$ and $y_0$ are same. So there is unique $y_0$ for each $x_0$ such that $dist(x_0,y_0)=a$. 
\end{minipage}
\begin{minipage}{0.35\textwidth}
\centering
\includegraphics[width=2.4in,keepaspectratio]{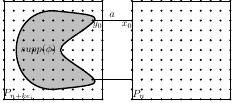}
\end{minipage}
\\But this implies:
$$\dprod{\delta_{x_0}}{(H^\omega)^a \phi}=\dprod{\delta_{x_0}}{(\Delta)^a \phi}=\dprod{\delta_{y_0}}{\phi}\neq 0.$$
Hence we get the contradiction.

Next we show that the property $\mathbb{P}[\Omega_{q,r}]=1$ for $q,r\in\ZZ^d$ is transitive. That is if $\mathbb{P}[\Omega_{q,p}]=1$ and $\mathbb{P}[\Omega_{p,r}]=1$, then
$$\mathbb{P}[\Omega_{q,r}]=1.$$
Using \cite[Lemma 3.1]{AM1} we have $P_q(H^\omega-z)^{-1}P_p$ and $P_p(H^\omega-z)^{-1}P_r$ are invertible for almost all $z\in\CC^{+}$ w.r.t Lebesgue measure. Denote $G^{\lambda}_{s,t}(z)=P_s(H^\omega+\lambda P_p-z)^{-1}P_t$, then using resolvent equation we have
$$G^{\lambda}_{q,r}(z)=G^{0}_{q,r}(z)-\lambda G^{0}_{q,p}(z)(I+\lambda G^{0}_{p,p}(z))^{-1}G^{0}_{p,r}(z).$$
Looking at determinant, we have (using the notation $l=rank(P_p)$)
\begin{align*}
\det(G^{\lambda}_{q,r}(z))&=\det(G^{0}_{q,r}(z)-\lambda G^{0}_{q,p}(z)(I+\lambda G^{0}_{p,p}(z))^{-1}G^{0}_{p,r}(z))\\
&=\frac{\lambda^{l^2}\det(G^{0}_{q,p}(z)(G^{0}_{p,p}(z))^{l-1}G^{0}_{p,r}(z))+\sum_{i=1}^{l^2-1}a_i(z)\lambda^i}{\det(I+\lambda G^{0}_{p,p}(z))}
\end{align*}
Since $G^0_{p,p}(z)$ is invertible for all $z\in\CC^{+}$, using \cite[Corollary 2.3]{AM1} we have $\det(G^{\lambda}_{q,r}(z))\neq 0$ for almost all $z\in\CC^{+}$, for almost all $\lambda$ w.r.t Lebesgue measure. 
So we have invertibility of $P_q(H^\omega+\lambda P_p-z)^{-1}P_r$ for almost all $\lambda$ w.r.t Lebesgue measure, which implies $Q^{\tilde{\omega}^\lambda}_q P_r$ has same rank as $P_r$ (here we use the notation $\tilde{\omega}^\lambda$ to denote the point in sample space $\tilde{\omega}^\lambda_k=\omega_k$ for $k\neq p$ and $\tilde{\omega}^\lambda_p=\omega_p+\lambda$) for almost all $\lambda$ w.r.t Lebesgue measure.
Since $\mathbb{P}[\Omega_{q,p}\cap\Omega_{p,r}]=1$, above argument implies
$$\mathbb{P}[\Omega_{q,r}]=1.$$

Using $\mathbb{P}[\Omega_{p,p+ke_i}]=1$ and transitivity inductively, we have the desired result.

\end{proof}
Following lemma is used to construct the eigenvalues $\{\tilde{E}_{n_1,\cdots,n_d}\}_{1\leq n_i\leq l_i}$ in lemma \ref{lemMain}. The matrix $A^{\omega,\lambda}_r$ in lemma \ref{lemMain} has the structure $\sum_i I^{i-1}\otimes D_r^{a_i,b_i}\otimes I^{d-i}$, and so the eigenvalues obtained here can be directly used. The eigenvalues are computed with the assumption that $r\gg \max\{|a_i|,|b_i|,1\}$. In this lemma we have written the eigenvalues upto $r^{-1}$ order term so that the simplicity of eigenvalues of $A^{\omega,\lambda}_r$ becomes clear. 
\begin{lemma}\label{lemEigAprx}
Let $r\gg\max\{|a|,|b|,1\}$, then the eigenvalues of the matrix 
$$D^{a,b}_r=r^2\Delta_l+(a+r)\proj{\delta_1}+(b+r)\proj{\delta_l}$$
where $\Delta_l$ is discrete laplacian on $\ell^2(1,\cdots,l)$, is given by
\begin{align}\label{lemEigAprxRes1}
\tilde{E}_n^{a,b,r}&=2r^2\cos\frac{\pi n}{l+1}+ \frac{4r}{l+1}\sin^2\frac{\pi n}{l+1}+\frac{2(a+b)}{l+1}\sin^2\frac{\pi n}{l+1}\nonumber\\
&\qquad\qquad+4C_n+\frac{4C_n}{r}(a+b)-\frac{D_n}{r}+O\left(\frac{1}{r^2}\right)
\end{align}
for $1\leq n\leq l$. Here 
$$C_n=\frac{2}{(l+1)^2}\sin^2\frac{\pi n}{l+1}\sum_{\substack{m\neq n\\ m\equiv n mod 2}}\frac{\sin^2\frac{\pi m}{l+1}}{\cos\frac{\pi m}{l+1}-\cos\frac{\pi n}{l+1}}$$
and $|D_n|<16(l+1)^3$.
\end{lemma}
\begin{proof}
Since this matrix is tri-diagonal, any eigenfunction has support on $\{1,l\}$. Set $\phi_n(x)=\sqrt{\frac{2}{l+1}}\sin\frac{\pi n x }{l+1}$ and note that
\begin{align*}
\norm{\left(D^{a,b}_r-2r^2\cos\frac{\pi n}{l+1}\right)\phi_n}_2^2&=\frac{2}{l+1}(a^2+b^2+2r^2+2r(a+b))\sin^2\frac{\pi n}{l+1}.
\end{align*}
Since $\phi_n$ is normalized, it implies that there exists an eigenvalue $E_n^{a,b,r}$ of $D^{a,b}_r$ such that 
\begin{equation}\label{lemEigAprxEq1}
 \left|E_n^{a,b,r}-2r^2\cos\frac{\pi n}{l+1}\right|<\frac{2\sqrt{1+\frac{a^2+b^2+2r(a+b)}{2r^2}}}{\sqrt{l+1}}r\leq 3r
\end{equation}
for large enough $r$ (i.e $\frac{|a|}{r},\frac{|b|}{r}\ll 1$).

Using resolvent equation we have
\begin{align*}
(D^{a,b}_r-z)^{-1}-(r^2\Delta_l-z)^{-1}=-(a+ r)(D^{a,b}_r-z)^{-1}\proj{\delta_1}(r^2\Delta_l-z)^{-1}\\
-(b+ r)(D^{a,b}_r-z)^{-1}\proj{\delta_l}(r^2\Delta_l-z)^{-1}
\end{align*}
Denote $\tilde{G}_{ij}(z)=\dprod{\delta_i}{(D^{a,b}_r-z)^{-1}\delta_j}$ and $G_{ij}(z)=\dprod{\delta_i}{(r^2\Delta_l-z)^{-1}\delta_j}$. Then using above equation we have
\begin{align*}
\hspace{-0.5in}\matx{\tilde{G}_{11}(z) & \tilde{G}_{1l}(z) \\ \tilde{G}_{l1}(z) & \tilde{G}_{ll}(z)}=\matx{ G_{11}(z) & G_{1l}(z) \\ G_{l1}(z) & G_{ll}(z)}\left(\matx{1 & 0 \\ 0 & 1} +\matx{a+ r & 0 \\ 0 & b+ r}\matx{ G_{11}(z) & G_{1l}(z) \\ G_{l1}(z) & G_{ll}(z)}\right)^{-1}
\end{align*}
So the eigenvalues of $D^{a,b}_r$ are given by the roots of 
\begin{align*}
\det\left(\matx{1 & 0 \\ 0 & 1} +\matx{a+ r & 0 \\ 0 & b+ r}\matx{ G_{11}(z) & G_{1l}(z) \\ G_{l1}(z) & G_{ll}(z)}\right)=0\\
\Leftrightarrow\qquad 1+(a+b+2 r)G_{11}(z)+(a+ r)(b+ r)(G_{11}(z)^2-G_{1l}(z)^2)=0.
\end{align*}
For the last equation we use the fact that
$$G_{11}(z)=G_{ll}(z)=\frac{2}{l+1}\sum_{n=1}^l \frac{\sin^2\frac{\pi n}{l+1}}{2r^2\cos\frac{\pi n}{l+1}-z}$$
and 
$$G_{1l}(z)=G_{l1}(z)=\frac{2}{l+1}\sum_{n=1}^l \frac{(-1)^n\sin^2\frac{\pi n}{l+1}}{2r^2\cos\frac{\pi n}{l+1}-z}.$$
We will use $a_n=\frac{2}{l+1}\sin^2\frac{\pi n}{l+1}$ and $E_n=2\cos\frac{\pi n}{l+1}$ to simplify the notations. First observe that
\begin{align*}
G_{11}(z)^2-G_{1l}(z)^2&=4\left(\sum_{\substack{n=1,\cdots,l\\ n\equiv0mod 2}}\frac{a_n}{r^2 E_n-z}\right)\left(\sum_{\substack{n=1,\cdots,l\\ n\equiv1mod 2}}\frac{a_n}{r^2 E_n-z}\right)\\
&=4\sum_{\substack{n,m=1,\cdots,l\\n\not\equiv m mod 2}}\frac{a_na_m}{(r^2 E_n-z)(r^2 E_m-z)}\\
&=\frac{4}{r^2}\sum_{n=1}^l\left(\sum_{m\not\equiv n mod 2} \frac{a_m}{E_m-E_n}\right)\frac{a_n}{r^2 E_n-z}.
\end{align*}
Set $c_n=\sum_{m\not\equiv n mod 2} \frac{a_m}{E_m-E_n}$. Combining all these give us
\begin{align*}
1+(a+b+2 r)\sum_{n=1}^l \frac{a_n}{r^2 E_n-z}+4\frac{(a+ r)(b+ r)}{r^2}\sum_{n=1}^l \frac{a_n c_n}{r^2 E_n-z}=0\\
1+\sum_{n=1}^l \frac{a_n}{r^2 E_n-z}\left(a+b+2 r+4c_n+\frac{4c_n(a+b)}{r}+\frac{4c_n ab}{r^2}\right)=0.
\end{align*}
Since we know $|E^{a,b,r}_{m}-r^2 E_m|<3r$ we have
\begin{align*}
0=1+\frac{a_m}{r^2 E_m-E^{a,b,r}_m}\left(2 r+a+b+4c_m+\frac{4c_m(a+b)}{r}+\frac{4c_m ab}{r^2}\right)\\
+\sum_{n\neq m} \frac{a_n}{r^2 E_n-E^{a,b,r}_m}\left(2 r+a+b+4c_n+\frac{4c_n(a+b)}{r}+\frac{4c_n ab}{r^2}\right)\\
\left[1+2r\frac{a_m}{r^2 E_m-E^{a,b,r}_m}\right]+\left[\frac{a_m(a+b+4c_m)}{r^2 E_m-E^{a,b,r}_m}+2r\sum_{n\neq m}\frac{a_n}{r^2 E_n-E^{a,b,r}_m}\right]\\
+\left[\frac{a_m}{r^2 E_m-E^{a,b,r}_m}\frac{4c_m(a+b)}{r}+\sum_{n\neq m}\frac{a_n(a+b+4c_n)}{r^2 E_n-E^{a,b,r}_m}\right]+O\left(\frac{1}{r^3}\right)=0.
\end{align*}
Here the brackets are arranged as per order terms of $r$. So we have the recurrence relation
\begin{align*}
E^{a,b,r}_m&=r^2 E_m+2ra_m+a_m(a+b+4c_m)+2r\sum_{n\neq m}a_n\frac{r^2 E_m-E^{a,b,r}_m}{r^2 E_n-E^{a,b,r}_m}\\
&+\frac{4a_mc_m(a+b)}{r}+\sum_{n\neq m}a_n(a+b+4c_n)\frac{r^2 E_m-E^{a,b,r}_m}{r^2 E_n-E^{a,b,r}_m}+O\left(\frac{1}{r^2}\right)
\end{align*}
Using $E^{a,b,r}_m=r^2 E_m+2ra_m+a_m\left(a+b+4c_m-4\sum_{n\neq m}\frac{a_n}{E_n-E_m}\right)+O(\frac{1}{r})$ on the above recursion we get
\begin{align*}
\hspace{-1cm}E^{a,b,r}_m&=r^2 E_m+2ra_m+a_m\left(a+b+4c_m-4\sum_{n\neq m}\frac{a_n}{E_n-E_m}\right)\\
&-4a_m\sum_{n\neq m}\frac{a_n}{E_n-E_m}\left(\frac{1}{2r}\left(a+b+4c_m-4\sum_{n\neq m}\frac{a_n}{E_n-E_m}\right)+\frac{2a_m}{r(E_n-E_m)}\right)\\
&+\frac{4a_mc_m(a+b)}{r}-\frac{2a_m}{r}\sum_{n\neq m}\frac{a_n(a+b+4c_n)}{E_n-E_m}+O\left(\frac{1}{r^2}\right).
\end{align*}
So rearranging the terms appropriately give us
\begin{align*}
E^{a,b,r}_m=& r^2 E_m+2ra_m+a_m(a+b)+4a_m\left(c_m-\sum_{n\neq m}\frac{a_n}{E_n-E_m}\right)\\
&+\frac{4a_m(a+b)}{r}\left(c_m-\sum_{n\neq m}\frac{a_n}{E_n-E_m}\right)\\
&-\frac{4a_m}{r}\left(\sum_{n\neq m}\frac{a_n}{E_n-E_m}\left(2c_m+2c_n-2\sum_{n\neq m}\frac{a_n}{E_n-E_m}+\frac{2a_m}{E_n-E_m}\right)\right)\\
&\qquad +O\left(\frac{1}{r^2}\right).
\end{align*}
Notice that \eqref{lemEigAprxRes1} is same as above equation where
\begin{align}\label{lemEigAprxEq2}
C_m&=a_m\left(c_m-\sum_{n\neq m}\frac{a_n}{E_n-E_m}\right)\nonumber\\
&=a_m\sum_{\substack{n\neq m\\ n\equiv m mod 2}}\frac{a_n}{E_n-E_m}\nonumber\\
&=\frac{2}{(l+1)^2}\sin^2\frac{\pi m}{l+1}\sum_{\substack{n\neq m\\ n\equiv m mod 2}}\frac{\sin^2\frac{\pi n}{l+1}}{\cos\frac{\pi n}{l+1}-\cos\frac{\pi m}{l+1}},
\end{align}
and
\begin{align}\label{lemEigAprxEq3}
D_m&=\sum_{n\neq m}\frac{8a_ma_n}{E_n-E_m}\left(c_m+c_n-\sum_{n\neq m}\frac{a_n}{E_n-E_m}+\frac{a_m}{E_n-E_m}\right)\nonumber\\
|D_m|&\leq \frac{128}{l+1}\left(\max_{n\neq m}\frac{1}{|E_n-E_m|}\right)^2\leq \frac{16}{l+1}\frac{1}{\sin^4\frac{\pi}{2(l+1)}}< 16(l+1)^3.
\end{align}
Hence completing the proof of the lemma.

\end{proof}

\appendix
\section{Appendix}
The lemma \ref{multLem1} relies heavily on the expression \eqref{singSpecEq1}. The result is valid for much larger class of Anderson type Hamiltonians. The theorem stated here is part of work done in the thesis \cite{AM2}.
The result is similar to the conclusion of Theorem 1.4 by Jak\v{s}i\'{c}-Last\cite{JL2}, but in this case the condition $\mathbb{P}(\omega: Q^\omega_n P_m\text{ has full rank})=0$ does not imply that the Hilbert subspaces $\Hi^\omega_n$ and $\Hi^\omega_m$ are orthogonal.

First few notations are needed, on the separable Hilbert space $\Hi$ set
\begin{equation}\label{randOp1}
 A^\omega=A+\sum_{n\in\mathcal{N}}\omega_n P_n,
\end{equation}
where $A$ is bounded self-adjoint operator, $\{P_n\}_{n\in\mathcal{N}}$ is a countable collection of rank $N$ projection such that $\sum_{n\in\mathcal{N}}P_n=I$ and $\{\omega_n\}_{n\in\mathcal{N}}$ are independent real random variables following absolutely continuous distribution with bounded support. Set $E^\omega$ to be the spectral projection for the operator $A^\omega$.
Set $\Hi^\omega_n$ to be the minimal closed $A^\omega$-invariant subspace containing the vector space $P_n\Hi$, and $Q^\omega_n$ be the canonical projection from $\Hi$ to $\Hi^\omega_n$. 
\begin{theorem}\cite[Theorem 4.1.1 (4)]{AM2}\label{singSubSpThm}
On the separable Hilbert space $\Hi$ let $A^\omega$ be described by \eqref{randOp1}. Set $E^\omega_{sing}$ to be the orthogonal projection onto the singular part of the spectral measure for the operator $A^\omega$.
If
\begin{equation}\label{condEq3}
 \mathbb{P}[\omega: Q^\omega_n P_m\text{ has same rank as }P_m]=1\qquad\forall n,m\in\mathcal{N},
\end{equation}
then almost surely
$$E^\omega_{sing}\Hi=E^\omega_{sing}\Hi^\omega_{n}$$
for any $n\in\mathcal{N}$.
\end{theorem}
\begin{proof}
Set $A^{\omega,\mu_1,\mu_2}=A^\omega+\mu_1 P_n+\mu_2 P_m$, then using \cite[lemma 3.4]{AM1} the condition \eqref{condEq3} implies that the matrix $P_n(A^{\omega,\mu,0}-x-\iota 0)^{-1}P_m$ and $P_n(A^{\omega,0,\mu}-x-\iota 0)^{-1}P_m$ are invertible for almost all $x$ (w.r.t Lebesgue measure) almost surely for any $n,m\in\mathcal{N}$. Using lemma \ref{simLem1} we have
\begin{equation}\label{subspEq1}
 E^{\tilde{\omega}}_{sing}\Hi^{\tilde{\omega}}_{n}\subseteq E^{\tilde{\omega}}_{sing}\Hi^{\tilde{\omega}}_{m}~~\&~~E^{\tilde{\omega}}_{sing}\Hi^{\tilde{\omega}}_{m}\subseteq E^{\tilde{\omega}}_{sing}\Hi^{\tilde{\omega}}_{n},
\end{equation}
almost all $\mu_1,\mu_2$ (w.r.t Lebesgue measure), where $E^{\tilde{\omega}}_{sing}$ is the orthogonal projection onto the singular part of the spectral measure for $A^{\omega,\mu_1,\mu_2}$ and $\Hi^{\tilde{\omega}}_i$ is the minimal closed $A^{\omega,\mu_1,\mu_2}$-invariant subspace containing $P_i\Hi$.

So using condition \eqref{condEq3} and \eqref{subspEq1} gives us
$$E^\omega_{sing}\Hi^\omega_{n}=E^\omega_{sing}\Hi^\omega_{m}\qquad\forall n,m\in\mathcal{N},$$
which implies
$$E^\omega_{sing}\Hi=\cup_{n\in\mathcal{N}} E^\omega_{sing}\Hi^\omega_n=E^\omega_{sing}\Hi^\omega_m$$
for any $m\in\mathcal{N}$.

\end{proof}

\begin{lemma}\cite[lemma 4.3.10]{AM2}\label{simLem1}
On the Hilbert space $\Hi$ we have two rank $N$ projections $P_1,P_2$ and a self adjoint operator $H$. 
Set $H_\mu=H+\mu P_1$, $G_{ij}(z)=P_i(H-z)^{-1}P_j$ and $G_{ij}^\mu(z)=P_i(H_\mu-z)^{-1}P_j$; set 
$$S=\{x\in\RR|\text{Entries of $G_{ij}(x+\iota 0)$ exists and are finite }\forall i,j=1,2\}$$
and
$$S_{12}=\{x\in S| G_{12}(x+\iota 0)\text{ is invertible}\}.$$
Let $E^\mu_{sing}$ denote the orthogonal projection onto the singular part of spectral measure for $H_\mu$ and set $\Hi^\mu_{i,sing}$ denote the closed $E^\mu_{sing}H_\mu$-invariant linear subspace containing $P_i\Hi$. 
If $S_{12}$ has full Lebesgue measure, then $\Hi_{2,sing}^\mu\subseteq\Hi^\mu_{1,sing}$ for almost all $\mu$ (with respect to Lebesgue measure).
\end{lemma}
\begin{proof}
Let $\{e_{ij}\}_{j=1}^N$ be a basis of $P_i\Hi$ for $i=1,2$. In this basis the linear operators $G_{ij}^\mu(z)$ and $G_{ij}(z)$ are matrices. Using Poltoratskii's theorem (the version used here is \cite[Theorem 1.1]{JL3}) for the matrix case we have
$$\lim_{\epsilon\downarrow 0}\frac{1}{tr(G_{11}^\mu(x+\iota\epsilon))}G_{11}^\mu(x+\iota\epsilon)=M^\mu_1(x),$$ 
for almost all $x$ w.r.t. $\sigma_{1,sing}^\mu$ (here $\sigma_i^\mu$ denotes the trace measure $tr(P_i E^{H_\mu}(\cdot) P_i)$ and set $\sigma^\mu_{1,sing}$ to be singular part of the measure). Using non-negativity of the spectral measure we have $M_1^\mu(x)\geq 0$ for almost all $x$ with respect to $\sigma^\mu_{1,sing}$. 
Following the proof of lemma \cite[lemma 3.7]{AM1}, we get
$$\lim_{\epsilon\downarrow 0}\frac{1}{tr(G_{11}^\mu(x+\iota\epsilon))}G_{ii}^\mu(x+\iota\epsilon)=M^\mu_i(x)\geq 0$$ 
for almost all $x$ w.r.t. $\sigma_{1,sing}^\mu$. Let $U_i^\mu(x)$ be the unitary matrix such that $U_i^\mu(x) M_i^\mu(x) U_i^\mu(x)^\ast$ is diagonal with entries $f_{i1}^\mu(x),\cdots,f_{iN}^\mu(x)$ for $x$ in support of $\sigma_{1,sing}^\mu$ (by using Hahn-Hellinger Theorem, one can choose the $U_i^\mu(\cdot)$ to be Borel measurable function). 
For $x$ not in the support of $\sigma_{1,sing}^\mu$ set $U_{ij}^\mu(x)=0$ and define $\psi^\mu_{ij}=U_{ij}^\mu(H_\mu)^\ast e_{ij}$.

We observe that
\begin{align*}
&\dprod{\psi_{ij}^\mu}{(H_\mu-z)^{-1}\psi_{kl}^\mu}=\int \frac{1}{x-z}\dprod{\psi_{ij}^\mu}{E^{H_\mu}(dx)\psi^\mu_{kl}}\\
&\qquad=\int \frac{1}{x-z}\dprod{U_i^\mu(x)^\ast e_{ij}}{E^{H_\mu}(dx)U_k^\mu(x)^\ast e_{kl}}\\
&\qquad=\int \frac{1}{x-z}\sum_{p,q}\dprod{e_{ij}}{U_i^\mu(x) e_{ip}}\dprod{e_{kq}}{U_k^\mu(x)^\ast e_{kl}}\dprod{e_{ip}}{E^{H_\mu}(dx)e_{kq}}\\
&\qquad=\int \frac{1}{x-z}\sum_{p,q}\dprod{e_{ij}}{U_i^\mu(x) e_{ip}}\overline{\dprod{e_{kl}}{U_k^\mu(x) e_{kq}}}\dprod{e_{ip}}{E^{H_\mu}(dx)e_{kq}}.
\end{align*}
So as a consequence of Poltoratskii's theorem
\begin{align*}
&\lim_{\epsilon\downarrow0}\frac{\dprod{\psi_{ij}^\mu}{(H_\mu-x-\iota\epsilon)^{-1}\psi_{kl}^\mu}}{tr(G_{11}^\mu(x+\iota\epsilon))}\\
&\qquad=\sum_{p,q}\dprod{e_{ij}}{U_i^\mu(x) e_{ip}}\overline{\dprod{e_{kl}}{U_k^\mu(x) e_{kq}}}\left(\lim_{\epsilon\downarrow0}\frac{\dprod{e_{ip}}{(H_\mu-x-\iota\epsilon)^{-1}e_{kq}}}{tr(G_{11}^\mu(x+\iota\epsilon))}\right)
\end{align*}
Therefore for $j\neq k$ we have $\dprod{\psi^\mu_{ij}}{(H_\mu-z)^{-1}\psi^\mu_{ik}}=0$, because the normal limit to $\RR$ is zero for all $x$. 
But the measure $\dprod{\psi^\mu_{ij}}{E^{H_\mu}(\cdot)\psi^\mu_{ik}}$ cannot have any absolutely continuous component, because by the construction of $\{\psi^\mu_{pq}\}$, the measure $\dprod{\psi^\mu_{pq}}{E^{H_\mu}(\cdot)\psi^\mu_{pq}}$ is supported on $supp(\sigma^\mu_{1,sing})$ which has zero Lebesgue measure. 
So as a consequence of F. and M. Riesz theorem (the theorem used is \cite[Theorem 2.2]{JL2}) the Hilbert subspace $\Hi^\mu_{\psi^\mu_{ij}}$ is orthogonal to $\Hi^\mu_{\psi^\mu_{ik}}$, where $\Hi^\mu_{\phi}$ denotes the minimal closed $H_\mu$-invariant subspace containing $\phi$.

Using the steps of proof of lemma \cite[lemma 3.7]{AM1} we have
$$M_2^\mu(x)=\lim_{\epsilon\downarrow0}\frac{1}{tr(G_{11}^\mu(x+\iota\epsilon))}G_{22}^\mu(x+\iota\epsilon)=\mu^2 G_{12}(x+\iota 0)^\ast M_1^\mu(x) G_{12}(x+\iota 0)$$
for almost all $x$ w.r.t. $\sigma_{1,sing}^\mu$, hence giving us
$$f_{2i}^\mu(x)=\lambda^2 \sum_{j=1}^N \left|\dprod{\psi_{1j}^\mu}{G_{12}(x+\iota 0)\psi_{2i}^\mu}\right|^2 f_{1j}(x)$$
for a.e $x$ w.r.t $\sigma_{1,sing}^\mu$. This is important because 
\begin{align*}
 \dprod{\psi^\mu_{2i}}{g(H_\mu)\psi^\mu_{2i}}&=\lim_{\epsilon\downarrow0}\int g(x)\dprod{\psi^\mu_{2i}}{(H_\mu-x-\iota\epsilon)^{-1}\psi^\mu_{2i}}dx\qquad\qquad\forall g\in C_c(\RR)\\
&=\int g(x)f_{2i}^\mu(x)d\sigma_{1,sing}^\mu(x)\\
&=\lambda^2 \sum_{i=1}^N \int g(x)\left|\dprod{\psi^\mu_{2i}}{G_{12}(x+\iota 0)\psi^\mu_{2i}}\right|^2 f_{1j}(x)d\sigma_{1,sing}^\mu(x)
\end{align*}
for all $1\leq i\leq N$.
Using the resolvent identity we get
$$G_{12}^\mu(z)-G_{12}(z)=-\mu G_{11}^\mu(z)G_{12}(z)\qquad z\in\CC^{+},$$
which implies
$$\lim_{\epsilon\downarrow 0}\frac{1}{tr(G_{11}^\mu(x+\iota\epsilon))}G_{12}^\mu(x+\iota\epsilon)=-\mu M_1^\mu(x)G_{12}(x+\iota 0),$$
for almost all $x$ w.r.t. $\sigma^\mu_{1,sing}$, we have,
\begin{align*}
&\lim_{\epsilon\downarrow0}\frac{\dprod{\psi_{1j}^\mu}{(H_\mu-x-\iota\epsilon)^{-1}\psi_{2i}^\mu}}{tr(G_{11}^\mu(x+\iota \epsilon))}\\
&\qquad=\sum_{k,l}\dprod{e_{1j}}{U_1^\mu(x) e_{1k}}\overline{\dprod{e_{2i}}{U_2^\mu(x) e_{2l}}}\dprod{e_{1k}}{\left(\lim_{\epsilon\downarrow0}\frac{G_{12}^\mu(x+\iota \epsilon)}{tr(G_{11}^\mu(x+\iota\epsilon))}\right)e_{2l}}\\
&\qquad=-\mu\sum_{k,l}\dprod{e_{1j}}{U_1^\mu(x) e_{1k}}\overline{\dprod{e_{2i}}{U_2^\mu(x) e_{2l}}}\dprod{e_{1k}}{M_1^\mu(x)G_{12}(x+\iota 0)e_{2l}}\\
&\qquad=-\mu\dprod{e_{1j}}{U_1^\mu(x) M_1^\mu(x)G_{12}(x+\iota 0)U_2^\mu(x)e_{2i}}\\
&\qquad=-\mu f_{1j}^\mu(x)\dprod{\psi_{1j}^\mu}{G_{12}(x+\iota 0)\psi_{2i}^\mu}
\end{align*}
for almost all $x$ w.r.t $\sigma^\mu_{1,sing}$. On the support of $f^\mu_{1j}\sigma_{1,sing}^\mu$ set
$$\lim_{\epsilon\downarrow0}\frac{\dprod{\psi_{1j}^\mu}{(H_\mu-x-\iota\epsilon)^{-1}\psi_{2i}^\mu}}{\dprod{\psi_{1j}^\mu}{(H_\mu-x-\iota\epsilon)^{-1}\psi_{1j}^\mu}}=p_{ij}(x).$$
Because of Poltoratskii's theorem, the vector $p_{ij}(H_\mu)\psi^\mu_{1j}$ is the projection of $\psi^\mu_{2i}$ onto $E^\mu_{sing}\Hi^\mu_{\psi^\mu_{1j}}$. Finally for  almost all $x$ w.r.t. $f_{1j}^\mu d\sigma^\mu_{1,sing}$ we have
\begin{align*}
p_{ij}(x)&=\lim_{\epsilon\downarrow0}\frac{\dprod{\psi_{1j}^\mu}{(H_\mu-x-\iota\epsilon)^{-1}\psi_{2i}^\mu}}{\dprod{\psi_{1j}^\mu}{(H_\mu-x-\iota\epsilon)^{-1}\psi_{1j}^\mu}}\\
&=\lim_{\epsilon\downarrow0}\frac{\dprod{\psi_{1j}^\mu}{(H_\mu-x-\iota\epsilon)^{-1}\psi_{2i}^\mu}}{tr(G_{11}^\mu(x+\iota \epsilon))}\frac{tr(G_{11}^\mu(x+\iota \epsilon))}{\dprod{\psi_{1j}^\mu}{(H_\mu-x-\iota\epsilon)^{-1}\psi_{1j}^\mu}}\\
&=-\mu \dprod{\psi_{1j}^\mu}{G_{12}(x+\iota 0)\psi_{2i}^\mu}.
\end{align*}
Giving us
$$f_{2i}^\mu(x)=\sum_{j=1}^N |p_{ij}(x)|^2 f_{1j}^\mu(x)$$
for almost all $x$ w.r.t. $\sigma_{1,sing}^\mu$.
So multiplication by $p_{ij}$ is not only projection but also an isometry from $E^\mu_{sing}\Hi^\mu_{2i}$ to $\Hi^\mu_{1,sing}$. Since this is valid for all $\psi^\mu_{2j}$, we get
$$\Hi^\mu_{2,sing}\subseteq \Hi^\mu_{1,sing}$$
for almost all $\mu$ (with respect to Lebesgue measure).\\
\end{proof}

\bibliographystyle{plain}

\begin{thebibliography}{10}

\bibitem{AM}
Michael Aizenman and Stanislav Molchanov.
\newblock Localization at large disorder and at extreme energies: An elementary
  derivations.
\newblock {\em Communications in Mathematical Physics}, 157(2):245--278, 1993.

\bibitem{AW3}
Michael Aizenman and Simone Warzel.
\newblock Boosted simon-wolff spectral criterion and resonant delocalization.
\newblock {\em Communications on Pure and Applied Mathematics}, 2015.

\bibitem{BR1}
Robert~D Berman.
\newblock Some results concerning the boundary zero sets of general analytic
  functions.
\newblock {\em Transactions of the American Mathematical Society},
  293(2):827--836, 1986.

\bibitem{CH1}
J.M. Combes and P.D. Hislop.
\newblock Localization for some continuous, random hamiltonians in
  d-dimensions.
\newblock {\em Journal of Functional Analysis}, 124(1):149 -- 180, 1994.

\bibitem{GT1}
Fritz Gesztesy and Eduard Tsekanovskii.
\newblock On matrix--valued herglotz functions.
\newblock {\em Mathematische Nachrichten}, 218(1):61--138, 2000.

\bibitem{HK1}
Peter~D Hislop and M~Krishna.
\newblock Eigenvalue statistics for random schr{\"o}dinger operators with non
  rank one perturbations.
\newblock {\em Communications in Mathematical Physics}, 340(1):125--143, 2015.

\bibitem{JL1}
Vojkan Jak\v{s}i\'{c} and Yoram Last.
\newblock Spectral structure of anderson type hamiltonians.
\newblock {\em Inventiones mathematicae}, 141(3):561--577, 2000.

\bibitem{JL3}
Vojkan Jak\v{s}i\'{c} and Yoram Last.
\newblock A new proof of poltoratskii's theorem.
\newblock {\em Journal of Functional Analysis}, 215(1):103--110, 2004.

\bibitem{JL2}
Vojkan Jak\v{s}i\'{c} and Yoram Last.
\newblock Simplicity of singular spectrum in anderson-type hamiltonians.
\newblock {\em Duke Mathematical Journal}, 133(1):185--204, 05 2006.

\bibitem{KM2}
Abel Klein and Stanislav Molchanov.
\newblock Simplicity of eigenvalues in the anderson model.
\newblock {\em Journal of statistical physics}, 122(1):95--99, 2006.

\bibitem{AM1}
Anish Mallick.
\newblock Jak{\v{s}}i{\'c}--last theorem for higher rank perturbations.
\newblock {\em Mathematische Nachrichten}, 2015.

\bibitem{AM2}
Anish Mallick.
\newblock Spectral multiplicity for random operators with projection valued
  randomness.
\newblock {\em (thesis)}, 2016.

\bibitem{NNS}
Sergey Naboko, Roger Nichols, and G\"{u}nter Stolz.
\newblock Simplicity of eigenvalues in anderson-type models.
\newblock {\em Arkiv f\"{o}r Matematik}, 51(1):157--183, 2013.

\bibitem{POL1}
A.~G. Poltoratskii.
\newblock Boundary behavior of pseudocontinuable functions.
\newblock {\em Algebra i Analiz}, 5(2):189--210, 1993.

\bibitem{SSH}
Christian Sadel and Hermann Schulz-Baldes.
\newblock Random dirac operators with time reversal symmetry.
\newblock {\em Communications in Mathematical Physics}, 295(1):209--242, 2010.

\bibitem{BS2}
Barry Simon.
\newblock Cyclic vectors in the anderson model.
\newblock {\em Reviews in Mathematical Physics}, 6(05a):1183--1185, 1994.

\end{thebibliography}

\end{document}